\documentclass[a4paper]{amsart}

\usepackage[colorlinks=true, linkcolor=black, citecolor=red]{hyperref}
\usepackage{color}
\usepackage{mathtools} 
\usepackage{braket} 

\usepackage{verbatim,amsmath,amssymb,amscd,amsfonts,amsthm,enumitem} 
\usepackage[displaymath,pagewise]{lineno} 
\usepackage{stackrel}
\usepackage{color}
\usepackage{url}

\newtheorem{thm}{Theorem} \newtheorem{prop}[thm]{Proposition}
\newtheorem{lem}[thm]{Lemma} 
\theoremstyle{definition}
\theoremstyle{definition}\newtheorem{rem}[thm]{Remark}

\numberwithin{equation}{section} \numberwithin{thm}{section}

\newtheorem{definition}[thm]{Definition}

\newcommand*\patchAmsMathEnvironmentForLineno[1]{%
  \expandafter\let\csname old#1\expandafter\endcsname\csname
  #1\endcsname \expandafter\let\csname
  oldend#1\expandafter\endcsname\csname end#1\endcsname
  \renewenvironment{#1}%
  {\linenomath\csname old#1\endcsname}%
  {\csname oldend#1\endcsname\endlinenomath}}%
\newcommand*\patchBothAmsMathEnvironmentsForLineno[1]{%
  \patchAmsMathEnvironmentForLineno{#1}%
  \patchAmsMathEnvironmentForLineno{#1*}}%
\AtBeginDocument{%
  \patchBothAmsMathEnvironmentsForLineno{equation}%
  \patchBothAmsMathEnvironmentsForLineno{align}%
  \patchBothAmsMathEnvironmentsForLineno{flalign}%
  \patchBothAmsMathEnvironmentsForLineno{alignat}%
  \patchBothAmsMathEnvironmentsForLineno{gather}%
  \patchBothAmsMathEnvironmentsForLineno{multline}%
}

\newcommand{\dist}{{\rm dist}\,} 
 \newcommand{\R}{{\mathbb R}}
\newcommand{\N}{{\mathbb N}} 
 
 \newcommand{\spt}{{\rm spt}\,}
\newcommand{\diam}{{\rm diam}\,}

\newcommand{\E}{{\mathbb E}}

\newcounter{todocounter}

\newcommand{\eps}{\varepsilon}

 \newcommand{\var}{{\rm var}}
 
\DeclareMathOperator{\V}{\mathcal V^\infty}

\def\en{\mathbb N} \def\er{\mathbb R} \def\zet{\mathbb Z}
\def\qe{\mathbb Q}
\def\C{\mathcal C}  
  
\def\M{\mathcal{M}}  \def\I{\mathcal I}
\def\E{\mathcal E}  \def\J{\mathcal J}
\def\T{\mathcal T}  \def\eS{\mathfrak S}
\def\V{\mathcal V} \def\De{\mathfrak D}

\def\Ce{\textbf{C}}
\def\Pe{\textbf{P}}
\def\eS{\textbf{S}}
\def\eM{\textbf{M}}

\newcommand{\intr}{\operatorname{Int}}

\newcounter{defcounter} 

\newenvironment{measAssumption}{%
  \addtocounter{equation}{-1} \refstepcounter{defcounter}
  
  \begin{equation}} {\end{equation}}

\begin{document}

\title{On the asymptotics of counting functions for Ahlfors regular sets}
\author{Du\v san Pokorn\'y}

\author{Marc Rauch}

\thanks{During the work on the paper was the first author a junior
  researcher in the University Centre for Mathematical Modelling,
  Applied Analysis and Computational Mathematics (MathMAC). The research was also supported by grants GA\v CR~15-08218S and GA\v CR~18-11058S and by DFG}

\begin{abstract}

In this paper we deal with the so-called Ahlfors regular sets
(also known as $s$-regular sets) in metric spaces.
First we show that those sets correspond to a certain class of
tree-like structures.
Building on this observation we then study the following question:
under which conditions does the limit $\lim_{\eps\to 0+} \eps^s N(\eps,K)$
exist, where $K$ is an $s$-regular set and $N(\eps,K)$ is for instance
the $\eps$-packing number of $K$?

\end{abstract}

\email{dpokorny@karlin.mff.cuni.cz} \email{marc.rauch@posteo.de}

\keywords{Ahlfors regular, $s$-regular, packing number, Renewal theory, Minkowski measurability}
\subjclass[2010]{30L99, 28A80}
\date{\today}
\maketitle

\section{Introduction}

In \cite{Lalley1988}, Lalley proved as an application of the renewal theory the following result: 
If $K$ is a non-lattice self-similar set which satisfies the strong open set condition, then $\lim_{\eps\to 0+} \eps^D P(\eps,K)=P\in(0,\infty)$, where $D$ is the so-called Minkowski dimension of $K$.
Here $P(\eps,K)$ denotes the packing number of $K$ with parameter $\eps$.
Lalley continued with his fundamental paper \cite{Lalley1989}, where he proved abstract renewal theorems in shift spaces and applied those
to the study of the asymptotic behaviour of packing numbers for the limit sets of Shottky groups.

Such approaches based on the renewal theory turned out to be very fruitful in various contexts, including the most prominent case of the Minkowski measurability.
This includes for example results on the Minkowski measurability of self-similar sets by Gatzouras (see \cite{Gatzouras2001}) or a class of self conformal sets by Freiberg and Kombrink (see \cite{FreibergKombrink2011}).
Recently, very strong versions of of renewal theorems were obtained by Kombrink and Kesseb\"ohmer (see \cite{Kombrink2017} and \cite{KKFrobenius2017}) and applied in various situations in fractal geometry in \cite{KKApollonian2017}.
Similar use of the renewal theory can also be found when dealing with the so-called fractal curvatures (see e.g. \cite{Winter2006}).

In this paper we take another point of view and discuss how far those ideas can be pushed in the context of more general metric spaces.

The key notion of our contribution is the concept of an Ahlfors regular or, as we will call it, $s$-regular set.
A compact subset $K$ of a metric space $X$ is called $s$-regular, if there is a Borel measure $\mu$ and some
$0<\alpha,\beta,R<\infty$ such that $0<\mu(K)\leq\mu(X)<\infty$ and
\begin{align*}
\alpha r^s<\mu(B(x,r))<\beta r^s
\end{align*}
for all $x\in K$ and $R\ge r>0$. Many regular sets such as balls or spheres in $\er^d$, or, more general, compact domains with a $\C^2$ boundary (or their boundaries) are examples of Ahlfors regular sets. A classical examples of fractal Ahlfors regular sets are the self-similar sets satisfying the open set condition (see for instance \cite[Theorem~4.14]{Mattila}).

We first observe that $s$-sets can be characterized by the existence of a tree satisfying certain natural properties, which we call an $s$-tree. 
Assuming the existence of an $s$-tree which satisfies more restrictive conditions, we then prove results on the asymptotic behaviour of counting functions including
packing numbers, counting numbers and also results on Minkowski measurability.
In fact we prove an abstract result on an axiomatically defined class of counting functions, which then can be applied to the notions mentioned above.

The plan of the paper is as follows. 
First in Section~\ref{S:preliminaries} we recall some background material mostly from metric spaces, fractal geometry and ergodic theory, including the main points from the abstract renewal theory introduced in \cite{Lalley1989}.
We also introduce a concept of general counting function (Section~\ref{SS:countingFunction}). This captures (at least for our applications) all important features of functions such as covering or packing numbers.

In Section~\ref{S:sTrees} we introduce our key concept of an $s$-tree, a tree construction that allows us to characterize $s$-sets.
This is certainly not a new idea, such type of tree construction appeared for instance in \cite{Arcozzi2014} or \cite{Martina}, but to our knowledge, it is not known that such trees can be used to characterize $s$-sets.
The main points of this section are Lemma~\ref{L:sTreeImpliesRegular} and Lemma~\ref{L:RegularImpliessTree}, which explain how $s$-trees relate to $s$-sets.

Next Section~\ref{S:measurability} contains the main result (Theorem~\ref{MainTheoren}).
It states that under some technical but natural conditions 
the limit
$\lim_{\eps\to 0+} \eps^s C(\eps,K)$ exists, where $C$ is a general counting function and $K$ is an $s$-set.
Moreover, at the end of the section we show, how to modify the method so it can also be used for proving Minkowski measurability of $K$.

In the last Section~\ref{S:examples} we present a method how to produce trees satisfying the assumptions of our main theorem from Section~\ref{S:measurability}. Such trees are produced using a class of mappings which we call $\alpha$-almost similar mappings.
Moreover, we show that the class of $\alpha$-almost similar mappings somehow relates to the class of conformal $\C^{1+\alpha}$ diffeomorphisms (Proposition~\ref{P:C1,qtoqSmooth}) and so Theorem~\ref{MainTheoren} can be in particular applied to images of non-lattice self-similar sets with respect to conformal $\C^{1+\alpha}$ diffeomorphisms.

 \section{Preliminaries}\label{S:preliminaries}

\subsection{Notation and basic facts.}\label{SS:BasicNotions}
For $K\subseteq\er^d$ we denote by $|K|$ the volume (i.e. the $d$-dimensional Lebesgue measure) of $K$.
Let $(X,d)$ be a metric space.
By $\intr A$ and $\partial A$ we denote the interior and boundary of a set $A\subseteq X$,
respectively.
Given some $x\in X$ and $\eps>0$,
we denote by $B(x,\eps)\coloneqq\{y\in X:d(x,y)\le \eps\}$ the
\textbf{closed ball} around $x$ with radius $\eps$ 
and by  $U(x,\eps)\coloneqq\{y\in X:d(x,y)< \eps\}$ the
\textbf{open ball} around $x$ with radius $\eps$. 
If $\varnothing\neq A\subseteq X$ is some subset of $X$, the
\textbf{distance} of $x$ to $A$ is denoted by
$d(x,A)\coloneqq\inf_{y\in A}d(x,y)$.
Define the \textbf{parallel} set of a set $A$ with radius $\eps>0$ by
$$
A_{\eps}:=\{x\in X:\dist(x,A)\leq\eps\}.	
$$

Now let $ (Y,\rho) $ be another metric space and $\psi:A\to Y$ be a mapping
such that for some constants  $0<L_1,L_2<\infty$
\begin{align}\label{eq:biLipschitz}
L_1 d(x,y)\le \rho(\psi(x),\psi(y))\le L_2 d(x,y)
\end{align}
for every $x,y\in A$.
We call such $\psi$ to be \textbf{$L_1$-$L_2$-bi-Lipschitz} on $A$. 
Note that \eqref{eq:biLipschitz} immediately implies for each $K\subseteq X$ that
	\begin{align}\label{usefulbilipestimate}
	\left(\psi(K)\right)_{L_1\delta}\subseteq \psi(K_\delta)\subseteq \left(\psi(K)\right)_{L_2\delta},
	\end{align}
	whenever $\psi:X\to Y$ is onto and $L_1$-$L_2$-bi-Lipschitz on $X$.
	
Moreover, the condition that $\psi$ is onto, can be substituted by the following assumption (which is relevant when applying 
Theorem~\ref{MainTheoren}):
 suppose $\varnothing\neq K\subseteq \er^d$ to be compact, $\delta\geq 0$ and  $\psi:K_\delta\to\er^d$ to be $L_1$-$L_2$-bi-Lipschitz.
Then again \eqref{usefulbilipestimate} holds.
To see this, first observe that the second inclusion in \eqref{usefulbilipestimate} again follows directly from \eqref{eq:biLipschitz}.
For the first inclusion, note that the case
$\delta=0$ is trivial.
Next suppose $\delta>0$ and that there is some $y\in 	\left(\psi(K)\right)_{L_1\delta}\smallsetminus \psi(K_\delta)$.
Then there exists by the definition of the parallel set also an $x\in\left(\intr\left(\psi(K)\right)_{L_1\delta}\right)\smallsetminus \psi(K_\delta)$.
In particular we obtain $\dist(x,\psi(K))<L_1\delta$.
As $K$ is compact and $\psi$ is continuous, we can find a $u\in \psi(K)$ such that $\dist(x,\psi(K))=\left|x-u\right|$.
Let $L$ be the line segment between $x$ and $u$, and define the compact set $M\coloneqq L\smallsetminus \intr\psi(K_\delta)$.
Let $w$ be the nearest point to $u$ in $M$.
Then $w\in \partial \psi(K_\delta)$ and  $\dist(w,\psi(K))<L_1\delta$.
Define $z\coloneqq \psi^{-1}(w)$ (note that due to the compactness of $K_\delta$ one has $ \partial \psi(K_\delta)\subseteq \psi(K_\delta)$).
One has $z\in \partial K_\delta$. 
Indeed, suppose $z\in\intr K_\delta$.
By the Invariance of domain theorem, the set $\psi(\intr K_\delta)$ is open.
Thus $w=\psi(z)\in\psi(\intr K_\delta)\subseteq \intr \psi (K_\delta)$, which contradicts $w\in \partial \psi(K_\delta)$.

Now as  $z\in \partial K_\delta$, one has $\dist(z,K)=\delta$, but then \eqref{eq:biLipschitz}  implies
$\dist(w,\psi(K))\ge L_1\delta$, which contradicts  $\dist(w,\psi(K))<L_1\delta$.

\subsection{Minkowski dimension and measurability.}
\subsubsection{Packing and covering numbers.}\label{SS:PackingAndCovering}
Let $(X,d)$ be a metric space, $K\subseteq X$ and $\eps>0$.
A set $M\subseteq X$ will be called 
\begin{itemize}
\item \textbf{$\eps$-separated} in $ K $, if $M\subseteq K$ and $d(x,y)>\eps$ for all $x\neq y\in M$,
	\item \textbf{$\eps$-covering} of $K$, if $K\subseteq\bigcup_{x\in M} B(x,\eps)$,
	\item \textbf{$\eps$-packing} in $K$, if $M\subseteq K$ and $B(x,\eps)\cap B(y,\eps)=\varnothing$ for all $x\neq y\in M$.
\end{itemize}
To any of those notions (as well as some others) we can assign a corresponding counting function.
For instance we denote $\Pe(\eps, K)$ the maximal cardinality of an $\eps$-packing in $K$ and call it the \textbf{packing number} of $K$ (with parameter $\eps$).
Similarly $\Ce(\eps, K)$ will be the minimal cardinality of an $\eps$-covering of $K$ and is called the \textbf{covering number} of $K$ (with parameter $\eps$).
We define similarly to $\Pe(\eps, K)$ a counting function $\eS(\eps,K)$, which corresponds to the notion of $\eps$-separated sets. 

The above notions of counting functions also give rise to corresponding definitions of fractal dimensions.
For instance, we can define the \textbf{upper Minkowski dimension} $\overline{\dim}_\M(K)$ and the \textbf{lower Minkowski dimension} $\underline{\dim}_\M(K)$ of a
compact set $K\subseteq X$ using the packing numbers as
\begin{equation*}\label{MinDimension}
\overline{\dim}_\M(K)=\limsup_{\eps\to 0+} -\frac{\Pe(\eps, K)}{\log(\eps)}\quad\text{and}\quad \underline{\dim}_\M(K)=\liminf_{\eps\to 0+} -\frac{\Pe(\eps, K)}{\log(\eps)}.
\end{equation*}

In fact it is well known and easy to verify that the following chain of inequalities holds:
\begin{equation*}
\eS(2\eps,K)\leq \Pe(\eps, K)\leq \Ce(\eps, K)\leq \eS(\eps, K).
\end{equation*}
Thus we can replace $\Pe(\eps, K)$ with any other counting function and obtain exactly the same values for both dimensions.
If $\overline{\dim}_\M(K)=\underline{\dim}_\M(K)$, then the common value is denoted $\dim_\M(K)$ and called the \textbf{Minkowski dimension} of $K$.

If the set $ K $ has a well defined Minkowski dimension $ s $, one can ask about finer properties of the function $\eps\mapsto\eps^{s}\Pe(\eps, K)$.
For instance the finiteness of $\limsup_{\eps\to 0+}\eps^{s}\Pe(\eps, K)$, positivity of $\liminf_{\eps\to 0+}\eps^{s}\Pe(\eps, K)$, or, as we intend to do in Section~\ref{S:measurability} of this paper, the existence of the limit. 

It is good to note that the existence of the limit $\lim_{\eps\to 0+}\eps^{s}\Pe(\eps, K)$ is a relatively natural counterpart of the notion of Minkowski measurability in those metric spaces, where there is no natural notion of volume.

\subsubsection{General counting functions}\label{SS:countingFunction}
We want to generalize the notion of a counting function on a metric space $ M $.
Let $N$ be a mapping that assigns to each compact set $K\subseteq M$ and every $\eps>0$ a value $N(\eps,K)\in[0,\infty)$.
Consider the following conditions:
  \begin{enumerate}[label={\rm (C\arabic*)}]
  \item\label{C:monotonicity} $N(\eps,K)\leq N(\delta,K)$ whenever $\delta\leq\eps$;
  \item\label{C:monotonicityInInclusion} $N(\eps,K)\geq N(\eps,H)$ whenever $K\supseteq H$
  \item\label{C:subadititivity} $N(\eps, \bigcup_{i=1}^n K_i)\leq \sum\limits_{i=1}^{n}N(\eps, K_i)$, 
  for any compact sets $K_1,\dots,K_n$;
  \item\label{C:separation} there is a constant $A$ such that $N(\eps,K\cup Q)=N(\eps,K)+N(\eps, Q)$  whenever $\dist(K,Q)> A\eps$ ;
  \item\label{C:lipschitzImage} there is a constant $G\geq 0$ (depending on $N$) such that for every $\tau, L>0$ one has $N(\eps,K)\geq N(L\eps ,\phi(K))$ whenever $\phi:M\to M$ is a mapping which is $L$-Lipschitz on $K_{G\tau}$ and $\eps\leq\tau$;
  \item\label{C:comparability} there is a constant $B>0$ such that $\frac{1}{B}\eS(B\eps ,K)\leq N(\eps ,K)\leq B\eS(\frac{\eps}{B},K)$.
  \end{enumerate}
  Note that if both \ref{C:monotonicity} and \ref{C:comparability} hold, it follows that
  \begin{align}\label{B_bigger_than_1}
  B\ge 1.
  \end{align}
  
  It is easy to verify that $\Pe$, $\Ce$ and $\eS$ satisfy all of those conditions.
In fact conditions \ref{C:monotonicity}-\ref{C:subadititivity} and \ref{C:comparability} are trivial and it is also easy to see that \ref{C:separation} holds with $A=2$ for $\Pe$ and $\Ce$, and with $A=1$ for $\eS$.
The only non-trivial condition is \ref{C:lipschitzImage}, and we will provide a proof that it holds with $G=1$ for the packing number $\Pe$. The proof for $\Ce$ and $ \eS $ is similar, the constant $G$ for $\Ce$ and $\eS$ is equal to $2$ and $0$, respectively.

To do so let $\tau>0$ and let $\phi:M\to M$ be a mapping which is $L$-Lipschitz on $K_{\tau}$. Consider for some $\eps\leq\tau$ an $L\eps$-packing $x_1,\dots,x_n$ of $\phi(K)$ and find $y_i$ such that $x_i=\phi(y_i)$, $i=1,\dots,n$.
We will prove that $y_1,\dots,y_n$ is an $\eps$-packing of $K$.
Suppose that this is not the case.
Then there are $i$ and $j$ and $z\in M$ such that $z\in B(y_i,\eps)\cap B(y_j,\eps)$.
This in particular implies that $z\in K_\tau$ and therefore $\phi(z)\in B(x_i,L\eps)\cap B(x_j,L\eps)$ which is a contradiction to the fact that  $x_1,\dots,x_n$ form an $L\eps$-packing of $\phi(K)$.

\subsubsection{Minkowski measurability}\label{SS:minkowskiMeasurability}
Suppose that $K\subseteq \er^d$ is a compact set,
than there is another way of defining the Minkowski dimension, namely via the scaling of the volume of its parallel sets. 
More precisely
\begin{equation*}
\overline{\dim}_\M(K)=d-\limsup_{\eps\to 0+} \frac{\log(|K_\eps|)}{\log(\eps)}\quad\text{and}\quad \underline{\dim}_\M(K)=d-\liminf_{\eps\to 0+} \frac{\log(|K_\eps|)}{\log(\eps)}.
\end{equation*}
We can also define the upper and lower \textbf{$s$-Minkowski content} by
\begin{equation*}
\overline\M^s(K)=\limsup_{\eps\to 0+} \frac{|K_\eps|}{\eps^{d-s}}\quad\text{and}\quad \underline\M^s(K)=\liminf_{\eps\to 0+} \frac{|K_\eps|}{\eps^{d-s}}.
\end{equation*}
Those two notions are closely related to the Minkowski dimension by the fact that
\begin{equation*}
\overline{\dim}_\M(K)=\inf\{s: \overline\M^s(K)=0\}\quad\text{and}\quad \underline{\dim}_\M(K)=\inf\{s: \underline\M^s(K)=0\}.
\end{equation*}
The set $ K $ is called \textbf{Minkowski measurable}, if $\overline\M^s(K)=\underline\M^s(K)\in(0,\infty)$ for some $s$. The common value is then called \textbf{Minkowski content} of $K$ and denoted by $\M^s(K)$.
Note that the value of $s$ is in that case necessarily equal to the Minkowski dimension of $K$.

Define $\eM(\eps,K)=\frac{|K_\eps|}{\eps^d}$. 
Then $K$ is Minkowski measurable if and only if 
\begin{equation*}
\lim_{\eps\to 0+}\eps^s \eM(\eps,K)\in(0,\infty),
\end{equation*}

where $s$ is again the Minkowski dimension of $K$.
The mapping $(\eps,K)\to \eM(\eps.K)$ satisfies all conditions \ref{C:monotonicity}-\ref{C:comparability} but \ref{C:lipschitzImage}.
Indeed, the only non-trivial ones are \ref{C:monotonicity} and \ref{C:comparability}.
To see \ref{C:comparability} it is sufficient to observe that
\begin{equation*}
V_d \eS(\eps,K)\leq \eM(\eps,K)\leq 2^d V_d \Ce(\eps,K),
\end{equation*}
where $V_d$ is a volume of the $d$-dimensional unit ball.
Condition \ref{C:monotonicity} is a consequence of the following lemma:

\begin{lem}\label{rataj-lemma}
Let $K\subseteq\er^d$ be a non-empty compact set.
Then the function 
\begin{equation}
\eps\mapsto\frac{|K_\eps|}{\eps^d}
\end{equation}
is non-increasing on $(0,\infty)$.
\end{lem}
\begin{proof}
Let $f:(0,\infty)\to\er$ be defined as $f(r)=|K_r|$.
It is well known (see e.g. \cite[Lemma~2, Theorem~1]{Stacho1976}), that $f'_+$ exists everywhere in $(0,\infty)$ and that $f(r)=\int_0^r f'_+(t)\;dt$.
Moreover, the function $r\to\frac{f'_+(r)}{r^{d-1}}$ is non-increasing.
Those facts imply
\begin{equation}\label{eq:kneser1}
f(r)=\int_0^r f'_+(t)\;dt\geq \int_0^r f'_+(r)\left(\frac{t}{r}\right)^{d-1}\;dt=\frac{r}{d}f'_+(r).
\end{equation}
Having now $0<r<s$ we can write
\begin{equation}\label{eq:kneser2}
\begin{aligned}
f(s)-f(r)&=\int_r^s f'_+(t)\;dt
\leq \int_r^s f'_+(r)\left(\frac{t}{r}\right)^{d-1}\;dt \\
&=\frac{f'_+(r)}{r^{d-1}}\cdot \frac{s^d-r^d}{d}
\leq f(r)\left(\left(\frac{s}{r}\right)^d -1\right),
\end{aligned}
\end{equation}
where the last inequality holds by \eqref{eq:kneser1}.
Regrouping terms in \eqref{eq:kneser2} we obtain 
\begin{equation*}
f(s)\leq \left(\frac{s}{r}\right)^d f(r),
\end{equation*}
which is what we need.
\end{proof}

Although \ref{C:lipschitzImage} is not valid for $ \eM$, it satisfies a weaker condition 
\begin{enumerate}[label={\rm (C'\arabic*)}]
\setcounter{enumi}{4}
\item\label{C:biLipschitzImage} there is a constant $G\geq 0$ such that $N(\eps ,K)\geq\left(\frac{L}{M}\right)^d N(L\eps,\phi(K))$ whenever $\phi$ is a mapping which is $L$-$M$-bi-Lipschitz on $K_{G\tau}$ and $\eps\leq\tau$. 
\end{enumerate}
To see this pick $G=1$. Then each $\phi$ which is $L$-$M$-bi-Lipschitz on $K_{G\tau}$ is also $L$-$M$-bi-Lipschitz on $K_{\eps}$ for every $0<\eps\leq\tau$, and so for any such $\eps$ we can write
\begin{equation*}
\begin{aligned}
 \eM(\eps,K)&=\frac{|K_\eps|}{\eps^d}\geq\frac{|\phi(K_\eps)|}{M^d\eps^d}\geq\frac{|(\phi(K)_{L\eps})|}{M^d\eps^d}\\
&=\left(\frac{L}{M}\right)^d \frac{|(\phi(K)_{L\eps})|}{L^d\eps^d}=\left(\frac{L}{M}\right)^d  \eM(L\eps,\phi(K)),
\end{aligned}
\end{equation*}
where the first inequality holds due to the area formula for Lipschitz mappings,
and the second one due to \eqref{usefulbilipestimate}
\subsection{Self-similar sets}
Let $\varphi_1,\dots,\varphi_N:\er^d\to\er^d$ be contracting similarities with contracting ratios $0<r_1,\dots,r_N<1$.
It is well known (see e.g. \cite{Hutchinson1981}) that there exists a unique non-empty compact set $K$ satisfying
\begin{equation*}
K=\bigcup_{i=1}^N\varphi(K).
\end{equation*}
This is then called the \textbf{self-similar set} generated by similarities $\varphi_1,\dots,\varphi_N$ (the so called \textbf{iterated function system} - \textbf{IFS}).
The self-similar set $K$ (or rather the mappings $\varphi_1,\dots,\varphi_N$) is said to satisfy the \textbf{open set condition} (OSC), if there exists a non-empty open set $U\subseteq \er^d$ such that
\begin{itemize}
	\item $\varphi_i(U)\subset U$ for every $i=1,\dots,N$,
	\item $\varphi_i(U)\cap \varphi_j(U)=\varnothing$ whenever $i\not=j$.
\end{itemize}
The set $U$ is then called the \textbf{feasible open set} for $K$ (or for the mappings $\varphi_1,\dots,\varphi_N$).

If (OSC) is satisfied, the unique solution $s$ of the equation $\sum r_i^s=1$ (the so called \textbf{similarity dimension} of $K$) is then equal to both Hausdorff and Minkowski dimension of $K$.
Note, that by \cite{Schief1994} the open set condition is equivalent to the \textbf{strong open set condition} (SOSC), where one additionally assumes that $K\cap U\not=\varnothing$.

The self-similar set $ K $ (or rather the ratios $r_1,\dots,r_N$) is called \textbf{lattice}, if there exists some $r>0$ such that $\frac{\log(r_i)}{\log(r)}\in\en$ for every $i=1,\dots,N$.
It will be called \textbf{non-lattice}, if it is not lattice.

\subsection{Doubling and $s$-regular spaces}\label{SS:DoublinsandSRegular}
One of the most important properties of $s$-regular sets is the fact that their lower and upper Minkowski dimensions are both equal to $s$, and the same is true for the Hausdorff dimension. 
This can be formulated in a more quantitative way by observing that there are $0<L^{\pm}<\infty$ such that
\begin{equation}\label{sRegularInequalinies}
  \liminf_{\eps\to 0+} \eps^{s}N(\eps, K)=L^{-}\quad\text{and}\quad
  \limsup_{\eps\to 0+} \eps^{s}N(\eps, K)=L^{+}.
\end{equation}
Here the function $N$ can represent any counting function defined in
Section~\ref{SS:PackingAndCovering}.
\begin{definition}
  A metric space $(X,d)$ is called \textbf{doubling}, if there exists
  a number $M\in \N$ such that for all $x\in X$ and all $r>0$ there
  exist $x_1,\dots,x_M\in X$ satisfying
  \begin{align*}
    B(x,2r)\subseteq \bigcup_{i=1}^M B(x_i,r).
  \end{align*}
\end{definition}
It is good to note that every $s$-regular set is also a doubling
metric space, this can be seen from the fact that the doubling
property is equivalent to the existence of a so-called doubling
measure.

We conclude the section with the following simple observation
concerning $s$-regular sets.
\begin{lem}\label{L:differentEpsilons}
  Let $K$ be an $s$-regular set and $\eps>0$.  Then there is a constant
  $R=R(\eps)\in (0,\infty)$ such that
	\begin{equation}
	\frac{1}{R}\left(\frac{\eps_{1}}{\eps_{2}}\right)^s\leq \frac{\eS(\eps_{2}, K)}{\eS(\eps_{1}, K)}\leq R\left(\frac{\eps_{1}}{\eps_{2}}\right)^s,
	\end{equation}
	whenever $\eps\ge\eps_{1}, \eps_{2}>0$. 
\end{lem}
\begin{proof}
Fix $ \eps>0 $.
From observation \eqref{sRegularInequalinies} and the fact that  $\eps\mapsto \eS(\eps,K)$ is positive and bounded on any interval not containing $0$, it follows that there is an $ C=C(\eps)\in (0,\infty) $ such that
\[
\frac{1}{C} L^-\le \eS(\eps',K)\le C \,L^+
\]
for each $ 0<\eps'\le \eps $.
Thus, if $  0<\eps_{1}, \eps_{2}\le \eps$, one has
\begin{equation*}
\frac{\eS(\eps_{2}, K)}{\eS(\eps_{1}, K)}=\left(\frac{\eps_{1}}{\eps_{2}}\right)^s\frac{\eps_{2}^s \eS(\eps_{2}, K)}{\eps_{1}^s \eS(\eps_{1}, K)}\leq C^2\;\frac{L^+}{L^-}\left(\frac{\eps_{1}}{\eps_{2}}\right)^s,
\end{equation*}
and similarly
\[
\frac{L^-}{C^2L^+}\left(\frac{\eps_{1}}{\eps_{2}}\right)^s	\le\frac{\eS(\eps_{2}, K)}{\eS(\eps_{1}, K)}.
\]
Setting $ R(\eps)\coloneqq C^2\frac{L^+}{L^-} $ shows the statement.
\end{proof}

\subsection{Shift space and shifts of finite type}
For $N\ge 2$, let $A_N\coloneqq\{0,\dots\,N-1\}$ be the
\textbf{alphabet} and $\Sigma_N\coloneqq A_N^\N$ be the \textbf{space
  of addresses} with $N$ symbols.  Given some \textbf{address}
$\omega=(\omega_1,\omega_2,\dots)\in \Sigma_N$, we use the notation
$\omega=\omega_1\omega_2\dots$.  Similarly, if we have some $k\ge 0$,
we denote the first $k$ symbols of $\omega$ by
$\omega|_k\coloneqq \omega_1\omega_2\dots\omega_k\in
\{0,\dots\,n-1\}^k$. In case $k=0$ one has $\omega|_k\coloneqq\varnothing$ for
all $\omega\in\Sigma_N$, which is called \textbf{empty word}.  If
$\mathcal{T}\subseteq \Sigma_N$ is a subset,
$\mathcal{T}^*\coloneqq\{\omega|_k: \omega\in\mathcal{T}, k\ge 0\}$
denotes the \textbf{set of all finite words} of $\mathcal{T}$.  To
distinguish between members of $\mathcal{T}$ and $\mathcal{T}^*$, we
use capital letters for finite words $I\in\mathcal{T}^*$, and Greek
lower-case letters for addresses $\omega\in\mathcal{T}$.

Given some word $I\in\mathcal{T}^*$, there exists by definition an
address $\omega\in\mathcal{T}$ such that $\omega|_k=I$ for some
$k\ge 0$. As $k$ does not depend on the choice of $\omega$,
$|I|\coloneqq k$ is well defined and called \textbf{length} of $I$.
The length of an address $\omega\in\mathcal{T}$ is set to
$|\omega|\coloneqq\infty$.  Similarly to addresses we define
$I|_k\coloneqq I_1I_2\dots I_k$ for all $I\in\mathcal{T}^*$ and
$0\le k\le |I|$.  We have
$\mathcal{T}^*=\bigcup_{k=0}^\infty \mathcal{T}_k$, where
$\mathcal{T}_k\coloneqq\{I: I\in\mathcal{T}^*, |I|=k\}$ denotes the
\textbf{set of finite words of length $k$}.  As a special case,
$\mathcal{T}_1\subseteq A_N$ for all $\mathcal{T}\subseteq \Sigma_N$
follows.

Given $x,y\in\mathcal{T}^*\cup\mathcal{T}$ such that $|x|\le |y|$, we
say $x$ and $y$ are \textbf{incomparable}, if $x_k\neq y_k$ for some
$1\le k \le |x|$. We write $x\prec y$, if $x_k=y_k$ for all
$1\le k\le |x|$. If $x\prec y$ does not hold we write $x\nprec y$. The \textbf{concatenation} $IJ$ of $I,J\in\Sigma_N$
is defined to be
$IJ\coloneqq I_1I_2\dots I_{|I|}J_1J_3\dots J_{|J|}\in\Sigma_N^*$.  We
agree that concatenation has higher precedence than the length
operator; that is $IJ|_k = (IJ)|_k$ for all $I,J\in\mathcal{T}^*$ and
$k\ge 0$.
For each $I\in\mathcal{T}^*$, the set
$[I]\coloneqq\{\omega\in\mathcal{T}:\omega_k=I_k \text{ for }
k=1,\dots,|I|\}$ is called \textbf{cylinder} of the word $I$.  We set
$[I]\coloneqq \mathcal{T}$, if $I$ is the empty word in $\mathcal{T}$.

We equip $\Sigma_N$ with the product topology, which is induced for
example by the metric $(\omega,\omega')\mapsto \frac{1}{2^n}$, where
$n\coloneqq\max\{k\ge 0:\omega|_k=\omega'|_k\}$.  Thus $\Sigma_N$ is a
compact metric space.  The mapping $\sigma_N:\Sigma_N\to\Sigma_N$,
$\omega=\omega_1\omega_2\dots\mapsto\omega_2\omega_3\dots $ is called
\textbf{full shift} on $\Sigma_N$. It is continuous with respect to
the product topology.
We will simply write $\sigma$ in the cases when the subscript $N$ will be clear form the context.

Let $\T\subseteq \Sigma_N$ be an $ \sigma $-invariant subset, that is, $ \sigma^{-1}(\T)=\T $.
We say that $(\T,\sigma)$ is a \textbf{shift of finite type}, if there is an irreducible and aperiodic matrix $A\subseteq \{0,1\}^{N\times N}$ (the so-called \textbf{transition matrix})
such that
\begin{equation*}
\T=\left\{\omega\in\Sigma_N: A_{\omega_{n},\omega_{n+1}}=1\quad\text{for all $n\in\en$} \right\}.
\end{equation*}
Note that for simplicity we write $ (\T,\sigma) $ instead of $ (\T,\sigma|_{\T}) $.
Now if $f:\T\to\er$ is a continuous function, denote for each $ n\ge0 $ by $S_nf:\T\to\er$ the mapping defined by
\begin{equation*}\label{eq:ergodicSum}
S_nf(\omega)\coloneqq f(\omega)+f(\sigma\omega)+\cdots+f(\sigma^{n-1}\omega).
\end{equation*}
Also, define for $n\ge 0$ the \textbf{$n$-th variation} of $f$ by 
\begin{equation*}
\var_n(f)\coloneqq\sup\{|f(\omega)-f(\tau)|: \omega\in\T, \omega|_n=\tau|_n \}.
\end{equation*}
Although  $ \var_n(f)$ depends on $ \T $, and this dependence is not reflected by the notation,
throughout the paper the subshift $ (\T,\sigma) $ should always be clear from the context.
For $0<\alpha<1$ we in addition define
\begin{equation*}
|f|_\alpha\coloneqq\sup_{n\geq 0}\frac{\var_n(f)}{\alpha^n},
\end{equation*}
and call $f$ \textbf{$\alpha$-H\"older continuous} if $|f|_\alpha<\infty$.

We say that functions $f,g:\T\to\er$ are \textbf{cohomologous}, if there is a continuous function $h:\T\to\er$ such that $f-g=h\circ\sigma-h$.
A function $f$ will be called \textbf{lattice}, if it is cohomologous to a function that takes values in a proper closed (additive) subgroup of $\er$.

If $\T\subseteq\Sigma_N$ is a subset, we can represent $\T\cup\T^*$ as a subset of $\Sigma_{N+1}$ by identifying $I\in\T^*$ with $\omega_I\coloneqq(I_1,I_2,\dots,I_{|I|},N,N,\dots)\in\Sigma_{N+1}$, 
which then naturally extends all the notions defined above to $\T\cup\T^*$.
In particular the following lemma holds, proof of which is straightforward:
\begin{lem}\label{compactnesslemma}
Let $m\ge 0$ and $\T\subseteq \Sigma_N$ be compact.
Then $ \{ I\in \T\cup \T^*:|I|\ge m\}$ is compact in $\Sigma_{N+1}$.
\end{lem}

\subsection{Renewal theorems} 
Assume $ (\T,\sigma) $ to be a shift of finite type.
Let $f_*,g_*$ be $ \alpha $-H\"older continuous on $\T\cup\T^*$ and let $f,g$ be their respective restrictions to $\T$.
Suppose that $f_*>0$ on $\T\cup\bigcup_{i=k}^{\infty}\T_i$ for some $k\ge 0$, and $g_*\geq0$, but not identically $0$, satisfying $\var_n(g_*)=0$ for some $n\in\N_0$.
Suppose furthermore that $f$ is non-lattice.
Let $G:\er\to\er$ be a nonnegative monotone function .
Define $P(L)\subseteq\er\times(\T\cup\T^*)$ by
\begin{equation}\label{P(L)}
P(L)\coloneqq \left\{(t,I): S_{|I|+1}f_*(IL)>t\geq S_{j}f_*(IL),\; j\leq |I|\right\}
\end{equation}
and $N_{G}:\T\cup\T^*\to\er$ by
\begin{equation}\label{N_G(a,L)}
N_{G}(a,L)\coloneqq \sum_{I} g_{*}(IL)G\left(S_{|I|+1}f_*(IL)-a\right)\chi_{P(L)}(a,I).
\end{equation}
The following renewal-type theorem was proved by S. Lalley in \cite{Lalley1989}.
\begin{prop}[\cite{Lalley1989}, Corollary~3.2]\label{renewal_theorem}
Under the conditions above one has
\begin{equation*}
N_{G}(a,\omega)\sim e^{a\delta}\int_{0}^{\infty} G(t)\;F(\omega,dt),
\end{equation*}
whenever $\omega\in\T\cup\bigcup_{i=k}^{\infty}\T_i$.
Here $\delta$ is the unique zero of the pressure function $t\mapsto p(-t f)$,
and $F$ is of the form 
\begin{equation*}
F(\omega,t)=C_*(\omega)\left(e^{\delta \min(t-f_*(\omega),0)}-e^{-\delta f_*(\omega)}\right),
\end{equation*} 
where $C_*$ is a positive continuous function on $\T\cup\T^*$ depending only on $f_*$ and $g_*$.
\end{prop}
Note that in \cite{Lalley1989} the result is stated only with $k=1$, but the exact same proof holds for any $k\ge 1$.

 \section{Ahlfors regular sets and their corresponding $s$-trees.}

\label{S:sTrees}

This section is devoted to the observation that the $s$-regular sets
can be characterized by the existence of a tree satisfying a simple
list of properties.  We start the section by defining this kind of
tree, which we will call an $s$-tree.

\begin{definition}
  Let $(X,d)$ be a metric space, $N\ge 2$, $s\in[0,\infty)$ and
  $\varnothing\not=\T\subseteq\Sigma_N$. Suppose there exist some constants
  $\rho,C,D\in(0,\infty)$ such that for every $I\in\mathcal{T}^*$
  there are $x_I\in X$ and $0<r_{I}<\infty$ satisfying the following
  properties:
  \begin{enumerate}[label={\rm (T\arabic*)}]
  \item\label{jedna} $d (x_I,x_J)\geq C(r_I+r_J)$ for all
    $J\in\mathcal{T}^*$ such that $I$ and $J$ are incomparable;
  \item\label{dva}
    $\diam(\{x_{IJ}: J\in\Sigma^*_N, IJ\in\mathcal{T}^*\})\leq Dr_I$;
    \item\label{tri} For all $n\geq 0$ one has
    \begin{equation*}
      \sum_{\substack{|J|=n\\ IJ\in\T^*}}r_{IJ}^s =r_I^s;
    \end{equation*}
  \item\label{ctyri} $r_I\to 0$ as $|I|\to\infty$;
  \item\label{pet} $r_{Ij}\geq\rho r_I$, if $j\in A_N$ and
    $Ij\in\mathcal{T}^*$.
  \end{enumerate}
  Then the triple
  $\mathrm{T}=(\T,\{x_I\}_{I\in\mathcal{T}^*},\{r_I\}_{I\in\mathcal{T}^*})$
  will be called \textbf{$s$-tree} in $X$.  A set $K\subseteq X$ defined
  by
  $K \coloneqq\{x_{\omega}:\omega\in\T, x_{\omega}\;\text{exists}\}$,
  where $x_\omega\coloneqq\lim_{n\to\infty}x_{\omega|_n}$, will be
  called the $s$-set \textbf{generated} by the $s$-tree $\mathrm{T}$.  Given
  some finite word $I\in\T^*$, we define
  $K_I\coloneqq\{x_\omega\in K:\omega\in\mathcal{T}, I\prec\omega\}$.
\end{definition}
\begin{rem}\label{kleinereins}
From \ref{ctyri} and \ref{pet} it follows immediately, that $0<\rho<1$.
From \ref{tri} it follows, that $ r_{IJ}\le r_I $ for all $ I,J\in\T^* $ such that $ IJ\in\T^* $.
In addition, we may always assume that $r_\varnothing=1$ and
\begin{align}\label{D_big_enough}
C\rho<D.
\end{align}
\end{rem}
\begin{rem}\label{R:SSStoStree}
  Note that the existence of an $s$-tree can be seen as a natural
  counterpart of an open set condition used e.g. in the context of
  self-similar sets.  Indeed, let $K$ be a self-similar set in $\er^d$
  generated by similarities $\varphi_1,\dots,\varphi_{N}$ with contraction
  ratios $r_i$. Let $U$ be a bounded feasible open set in (OSC) for
  this system of mappings.  Pick $x_{\varnothing}\in U$ arbitrary and
  put $D\coloneqq\diam U$, $C\coloneqq\dist(x_{\varnothing},U^c)$ and
  $\rho\coloneqq\min_{i}r_i$.  Define $r_I$ and $\varphi_I$ in the usual way for
  every $I\in\Sigma_N$, and put $x_I\coloneqq\varphi_I(x_{\varnothing})$.  Then it
  is easy to verify that the triple $(\Sigma_N,\{x_I\},\{r_I\})$ is an
  $s$-tree (with constants $C$, $D$ and $\rho$ as above) that
  generates $K$ is the sense of the above definition.

It may be also worth noting that there is also some very mild
counterpart to the strong open set condition in the sense that we can
additionally assume that $\{x_I\}_{I\in\T^*} \subseteq K$.  This
condition then implies that the overlaps of a two sets $K_I$ and $K_J$
with $I$ and $J$ incomparable are small which we will extensively use
in Section~\ref{S:measurability} (see Lemma~\ref{L:estimateOverlaps}).	

Note that in the self-similar case it is no problem to have condition $\{x_I\}_{I\in\T^*} \subseteq K$ satisfied.
To do this it is enough to consider $U$ to be feasible for (SOSC) and then pick $x_{\varnothing}\in K\cap U$.
\end{rem}

\begin{rem}\label{convention}
  From now on we will adopt convention that all concatenations will be
  assumed to be in $\T^*$.  For instance, condition \ref{tri} can then
  be written as $\sum_{|J|=n}r_{IJ}^s =r_I^s$ and condition \ref{pet}
  as $r_{Ij}\geq\rho r_I$ (without any additional assumptions).
\end{rem}

\begin{rem}\label{r:triAltToTri}
For $I\in\T^*$ and $n\in\en$ define $\T^n_I$ as the collection of all $\I\subseteq\bigcup_{k\ge n}\T_k$ satisfying
\begin{itemize}
\item for each $\omega\in\T$ there exists a unique $J\in\I$ such that $J\prec \omega$,
\item there is some $J\in\I$ such that $I\prec J$.
\end{itemize}
If $ \varnothing\neq \T\subseteq \Sigma_N $ is closed, condition \ref{tri} can be replaced by the following more technical, but also more flexible condition
\begin{enumerate}[label={\rm (T'\arabic*)}]
\setcounter{enumi}{2}
\item\label{triAlt}there is a constant $0<E<\infty$ such that for all $I\in\T^*$ one has
\begin{equation*}
\frac{1}{E} r_I^s\leq\sum_{IJ\in\I}r_{IJ}^s \leq E r_I^s
\end{equation*}
for each $\I\in\T^n_I$, $n\ge 0$.
\end{enumerate}

First we shall show, that \ref{tri} implies \ref{triAlt}.
We do so by construction a measure on $ \T $ with the help of \ref{tri}.
Define $ \mathcal{A}\coloneqq \{\varnothing\}\cup\{[I]:I\in\T^*\} $.
It is easy to see that $ \mathcal{A}$ is a semi-ring on $ \T $, that is one has 
\begin{itemize}
\item $ \varnothing\in \mathcal{A}$;
\item $ A,B \in \mathcal{A}$ implies $ A\cap B\in \mathcal{A} $;
\item  $ A,B \in \mathcal{A}$, then there exist finitely many pairwise disjoint $ C_i\in\mathcal{A} $ such that
$ B\smallsetminus A=\bigcup_{i=1}^nC_i $.
\end{itemize}
As $ \T $ is compact, each cylinder $ [I]\in\mathcal{A}  $ is open and closed in $ \T $.
Furthermore we have that $ \mathcal{A} $ is a generator of the Borel $ \sigma $-algebra of $ \T $.
Next define $ \nu(\varnothing)\coloneqq 0 $, $ \nu([I])\coloneqq r_I^s$ for all $ I\in\T^* $.
Clearly by \ref{tri} $ \nu $ is an additive, finite function on $ \mathcal{A} $.
As $ \mathcal{A} $ is a semi-ring, we also have that $ \nu $ is subadditive.
To show that $ \nu $ is $ \sigma $-subadditive, suppose $ A,A_i\in\mathcal{A} $ to be such that
$ A\subseteq\bigcup_{n=1}^\infty A_n$.
As $ \T $ is compact, there is a finite covering $ A\subseteq\bigcup_{i=1}^k A_{n_i}$, and the $\sigma $-subadditivity follows from the
subadditivity of $ \nu $.
Now using Carathéodory's extension theorem, there exists a unique measure $ \mu $ on $ \T $ such that $ \mu|_{\mathcal{A}}=\nu $.
From this it is easy to see that \ref{triAlt} is satisfied with constant $E= 1$. 

If $r_I$ satisfy \ref{triAlt}, then there are $\tilde r_I$ such that 
\begin{equation}\label{eq:comparable}
0<\frac{1}{E^{\frac{1}{s}}} r_I\leq \tilde r_I\leq E^{\frac{1}{s}} r_I<\infty
\end{equation}
for every $I\in\T^*$ and for which conditions \ref{jedna} to \ref{pet}  hold.
Those can be defined using
\begin{equation*}
\tilde r_I^s\coloneqq \sup_{m\ge 0}\left(\inf_{\I\in\T^m_I} \sum_{IJ\in\I} r_{IJ}^s\right)=\lim_{m\to\infty}\left(\inf_{\I\in\T^m_I} \sum_{IJ\in\I} r_{IJ}^s\right).
\end{equation*}
Note that above limit actually exists for fixed $I$, as $\T^{m}_I\supseteq \T^{m+1}_I$ for $m\ge 0$.

The validity of \eqref{eq:comparable} is clear from the definition.
Conditions \ref{jedna}, \ref{dva},\ref{ctyri} and \ref{pet} follow from \eqref{eq:comparable}.
It remains to verify \ref{tri}.
First we prove
\begin{equation}
 \sum_{\substack{|J|=n}}\tilde r_{IJ}^s \leq\tilde r_I^s.
\end{equation}
Suppose for contradiction that
\begin{equation*}
 \sum_{\substack{|J|=n}}\tilde r_{IJ}^s > \tilde r_I^s+\alpha.
\end{equation*}
for some $I$, $n$ and $\alpha>0$.
This implies that for every $m\ge 0$ there is an $\I^m\in\T^{m}_I$ such that
\begin{equation*}
\sum_{\substack{|J|=n}}\tilde r_{IJ}^s \geq \sum_{\substack{IJ\in\I^m}} r_{IJ}^s+\alpha.
\end{equation*}
Now define for each $m>|I|+n$ and $ J $ such that $IJ\in\T^*$ and $|J|=n$
\begin{equation*}
\I_J^m\coloneqq\{L\in\I^m: IJ\prec L\}.
\end{equation*}
By the definition of $\tilde r_{IJ}^s$ we have
\begin{equation*}
\liminf_{m\to\infty}\sum_{L\in\I_J^m} r_{L}^s \geq \tilde r_{IJ}^s
\end{equation*}
for every $J$.
Hence, for each $m_0>|I|+n$,
\begin{align*}
  \liminf_{m\to\infty}\sum_{\substack{IJ\in\I^m}} r_{IJ}^s 
&=\liminf_{m\to\infty}\sum_{\substack{|J|=n}}\sum_{L\in\I_J^m} r_{L}^s
\ge\sum_{\substack{|J|=n}}\liminf_{m\to\infty}\sum_{L\in\I_J^m} r_{L}^s \\
&\geq  \sum_{\substack{|J|=n}}\tilde r_{IJ}^s \geq \sum_{IJ\in\I^{m_0}} r_{IJ}^s+\alpha.
\end{align*}
Thus
\begin{align*}
  \liminf_{m\to\infty}\sum_{IJ\in\I^{m}} r_{IJ}^s\ge
    \liminf_{m\to\infty}\sum_{IJ\in\I^{m}} r_{IJ}^s+\alpha,
\end{align*}
which is not possible, as by \ref{triAlt} one has
\begin{align*}
0<\frac{1}{E}r^s_I\le \liminf_{m\to\infty}\sum_{IJ\in\I^{m}} r_{IJ}^s\le E r_I^s<\infty.  
\end{align*}

To prove the equality suppose for the contradiction that
\begin{equation*}
 \sum_{\substack{|J|=n}}\tilde r_{IJ}^s \le\tilde r_I^s-\alpha
\end{equation*}
for some $I$, $n$ and $\alpha>0$.
Let $k$ be the cardinality of the set $\{IJ\in\T^*: |J|=n\}$.
Due to the definition of $\tilde r^s_{IJ}$ we can find for every $m$ some $\I^m_{J}\in\T^m_{IJ}$ such that 
\begin{equation*}
\sum_{IJL\in\I^m_{J}} r_{IJL}^s-\frac{\alpha}{2k}\le\tilde r_{IJ}^s.
\end{equation*}
Consider now 
\begin{equation*}
\I^m\coloneqq \bigcup_{|J|=n} \{IJL\in\I^m_{J}\}.
\end{equation*}
Then
\begin{equation*}
\sum_{IJ\in\I^m} r_{IJ}^s-\frac{\alpha}{2}=\sum_{|J|=n}\left(\sum_{IJL\in\I^m_J} r_{IJL}^s-\frac{\alpha}{2k}\right)\le\sum_{|J|=n}\tilde r_{IJ}^s\le\tilde r_{I}^s -\alpha,
\end{equation*}
and so
\begin{equation*}
\tilde r_{I}^s+\frac{\alpha}{2}\leq \tilde r_{I}^s,
\end{equation*}
which is a contradiction to $0<\tilde r^s_{I}<\infty$.
\end{rem}

The section consist mainly of two lemmas, namely Lemma~\ref{L:sTreeImpliesRegular}, where we construct an $s$-regular
set given an $s$-tree and Lemma~\ref{L:RegularImpliessTree}, where
we assign an $s$-tree to every $s$-regular set.  This latter
observation is not really anything new, similar tree constructions
already appeared for instance in \cite{Arcozzi2014} and
\cite{Martina}.  The idea of the proof of
Lemma~\ref{L:sTreeImpliesRegular} is essentially the one of
\cite[5.3]{Hutchinson1981}.  In a way, it can be seen as pushing the
strategy of this proof to its boundaries.

However, this kind of equivalence between trees and $ s $-sets seems not to appear in the
literature, and for that reason we want to include the proofs here.

\begin{lem}\label{L:sTreeImpliesRegular}
Let $\mathrm{T}=(\T,\{x_I\}_{I\in\mathcal{T}^*},\{r_I\}_{I\in\mathcal{T}^*})$ be an
$s$-tree in a complete doubling metric space such that $ \T $ is closed.  Let $K$
be the $s$-set generated by $\mathrm{T}$.  Then $K$ is $s$-regular, and
$\mathrm{dim}_HK=\overline{\mathrm{dim}}_MK$.
\end{lem}

\begin{proof}
  First note that \ref{dva} implies
  \begin{equation}\label{diamKI}
    \diam(K_I)\leq Dr_I.
  \end{equation}
  As $X$ is complete, the mapping $\pi:\mathcal{T}\to K$,
  $\pi(\omega)\coloneqq x_\omega=\lim_{n\to\infty}x_{\omega|_n}$ is by \ref{dva} and \ref{ctyri} well-defined.
  In addition, $\pi$ is continuous.
  Let $\nu$ be the Borel measure such
  that $\nu([I])=r_I^s$ for all $I\in\mathcal{T}^*$ (see Remark \ref{r:triAltToTri} for the construction).
  Set $\mu\coloneqq\nu\circ \pi^{-1}$.
  Then
  \begin{equation}
    \label{eq:lower}
    \mu(K_I)\geq\nu([I])= r_I^s.
    \end{equation}
    
  First we prove that there is a constant $\alpha>0$ such that
  \begin{equation}\label{lower}
    \frac{\mu(B(x,r))}{r^s}\geq \alpha
  \end{equation}
  for every $x\in K$ and $\diam K>r>0$.
  To do so fix some $x=x_\omega\in K$.
  Using \eqref{diamKI} and \ref{ctyri}, there is a unique number $n\ge 1$ such
  that
  \begin{align}\label{sandwich_estimate}
  \diam(K_{\omega|_n})<r\le \diam(K_{\omega|_{n-1}}).
  \end{align}
  Put $I\coloneqq \omega|_n$.
  Then $x\in K_I$ and
  \begin{align*}
    \frac{\mu(B(x,r))}{r^s}&\stackrel{\eqref{sandwich_estimate}}{\geq}
    \frac{\mu(K_I)}{r^s}\stackrel{\eqref{eq:lower}}{\geq}
    \frac{r^s_I}{r^s}
    \stackrel{\ref{pet}}{\geq} \frac{\rho^s r^s_{I|_{n-1}}}{r^s}
    \stackrel{\ref{dva}}{\geq}\frac{\rho^s \diam^s(K_{I|_{n-1}})}{D^sr^s}\\
    & \stackrel{\eqref{sandwich_estimate}}{\geq} r^s\frac{\rho^s}{D^sr^s}=\frac{\rho^s}{D^s}\eqqcolon\alpha>0.
  \end{align*}
  
  Next we prove that there is a constant $\beta>0$ such that
  \begin{equation}\label{upper}
    \frac{\mu(B(x,r))}{r^s}\leq \beta
  \end{equation}
  for every $x\in K$ and $1>r>0$.
  For each $\omega\in\mathcal{T}$ let $p(\omega)$ be the smallest
  number satisfying $r_{\omega|_{p(\omega)}}<r$. Then using \ref{pet}
  one has for all $\omega\in\mathcal{T}$
  \begin{align}\label{smallest}
    \rho r\leq \rho r_{\omega|_{p(\omega)-1}}\leq r_{\omega|_{p(\omega)}}<
    r.
  \end{align}
  Define
  $\I\coloneqq\{\omega|_{p(\omega)}:\omega\in\mathcal{T}\}\subseteq\mathcal{T}^*$. Clearly
  $\I$ is at most countable and $[I]\cap[J]=\varnothing$ for
  $I\neq J\in\I$. Thus by setting
  $\mu_I\coloneqq \nu|_{[I]}\circ \pi^{-1}$ we obtain
  \begin{align*}
    \mu=\sum_{I\in\I}\mu_I.
    \end{align*}
     Fix $x=x_\omega\in K$.
  Let $\J$ be the set of those $I\in\I$ such that
  $K_I\cap B(x,r)\not=\varnothing$.  Pick $I$ in $\J$ and let
  $\tilde\omega$ be such that $x_{\tilde{\omega}}\in K_I\cap
  B(x,r)$,
  then
  \begin{align*}
  d (x_I,x_\omega)\leq d (x_I,x_{\tilde{\omega}})+d(x_{\tilde{\omega}},x_\omega)
  \stackrel{\ref{dva}}{\leq} Dr_{\tilde{\omega}|_{p(\tilde{\omega})}}+r\stackrel{\eqref{smallest}}{\leq} (D+1)r.
  \end{align*}
  Moreover, if $J\in\J$ such that $I\not=J$ then by \ref{jedna} and
  (\ref{smallest})
  \begin{align*}
    d (x_I,x_J)\geq C(r_I+r_J)\geq 2C\rho r.
  \end{align*}
  Thus $\{x_I:I\in\J\}$ forms a $C\rho r$-separated
  subset of $B(x,(D+1)r)$.
  Using \eqref{D_big_enough}, one can choose by \cite[Lemma
  3.3]{baloghrohner2007} some constant $\beta\in\en$ such that  $\#\J\le\beta$ and $\beta$
  is independent of $\J$, $r$ and $x$.
  Hence
  \begin{align*}
    \frac{\mu(B(x,r))}{r^s}=&\frac{\mu(B(x,r)\cap K)}{r^s}\leq \frac{\mu(\bigcup_{I\in\J}K_I)}{r^s}\\
    \leq&
          \frac{\sum_{I\in\J}\mu_I(K)}{r^s}=\frac{\sum_{I\in\J}r^s_I}{r^s}\stackrel{\eqref{smallest}}{\leq}\frac{\#\J r^s}{r^s}=\#\J\leq
          \beta.
  \end{align*}
  So $K$ is $s$-regular which in particular implies
  $\dim_{H} K=\overline{\dim}_{M} K=s$.
\end{proof}

\begin{lem}\label{L:RegularImpliessTree}
Let $(X,d)$ be a complete metric space.
Suppose $K\subseteq X$ to be a compact $s$-regular
set with respect to a Borel measure $ \mu $ with corresponding constants $ \alpha $
and $ \beta $. Then there exists an $s$-tree $ \mathrm{T} $ generating $K$.
\end{lem}

\begin{proof}
	Without any restriction we may assume that $\spt \mu=K$.
Recall for $ \eps>0 $ the packing number $\Pe(\eps, K)$, and denote for simplicity $ P(\eps)\coloneqq\Pe(\eps, K) $.
Fix $\frac{1}{6}>\delta>0$. For each $n\ge 0$, let
  $\V_{n}=\{x_1^n,\dots,x_{P(\delta^n)}^n\}$ be some maximal
  $\delta^n$-packing of $K$. We equip $\V_n$ with some strict, total
  order $\prec_n$, which means
  \begin{align*}
    x_1^n\prec_n\dots\prec_n x_{P(\delta^n)}^n
  \end{align*}
  for all $n\in\N$.  Fix some arbitrary $x^0\in K$ and define
  $\V_{0} \coloneqq\{x^0\}$. This in particular means that for
  $n\in\N$
  \begin{enumerate}
  \item\label{epsdense} every $\V_{n}$ is $2\delta^n$-dense,
  \item\label{epsseparated} every $\V_{n}$ is $\delta^n$-separated.
  \end{enumerate}
  Define $\V\coloneqq\bigcup_{n\ge0}\V_{n}$.  We consider $\V$ to be the set
  of vertices of a oriented tree $V$ with root $x^0$, where the set of
  edges $\E$ is defined in the following way: One has $(x,y)\in \E$ if
  and only if
  \begin{itemize}
  \item $x\in\V_n$ and $y\in\V_{n+1}$ for some $n\in\en_0$,
  \item $ d (x,y)= d (y,\V_n)$,
  \item if $ d (z,y)= d (y,\V_n)$ for some $z\in\V_n$, then
    $x\prec_n z$.
  \end{itemize}
  The last of the above properties ensure, that two distinct vertices
  in $\mathcal{V}_n$ cannot have a common successor. Hence, given
  $n>0$, for each $y\in\V_n$ there exists exactly one $x\in\V_{n-1}$
  such that $(x,y)\in\E$. Now assume $x,y\in\V_n$, $x\not=y$,
  $z\in\V_{n+1}$ and $(x,z)\in \E$. Then \eqref{epsdense} implies
  \begin{equation}\label{upperTnT}
    d (x,z)\leq 2\delta^n,
  \end{equation}
  and \eqref{epsseparated} implies
  \begin{equation}\label{lowerTnT}
    d (y,z)> \frac{\delta^n}{2}.
  \end{equation}
  Next note that the number off offsprings of any $x\in \V$ is bounded
  from above.  Indeed, let $x\in\V_{n}$ and $u_j\in\V_{n+1}$ for
  $j\in\Gamma$ be pairwise different offsprings of $x$.  Then by
  \eqref{upperTnT}
  \begin{equation}
    B\left(x,(2+\delta)\delta^n\right)\supseteq \bigcup_{j\in\Gamma}B(u_{i},\delta^{n+1}),
  \end{equation}
  and the balls $B(u_{i},\delta^{n+1})$ are pairwise disjoint.  Hence
  we can write
  \begin{equation}
    \beta \left(2+\delta\right)^{s}(\delta^{n})^{s}\geq \mu\left(B\left(x,(2+\delta)\delta^n\right)\right)
    \geq \sum_{j\in\Gamma}\mu(B(u_{i},\delta^{n+1}))\geq |\Gamma|\alpha \delta^s(\delta^{n})^{s},
  \end{equation}
  which implies
  \begin{align*}
    |\Gamma|\leq
    \frac{\beta}{\alpha}\cdot\left(\frac{2+\delta}{\delta}\right)^s\eqqcolon
    N.
  \end{align*}
  Also, each $x\in\V_n$ has at least one successor: As $\V_{n+1}$ is
  $2\delta^{n+1}$-dense in $K$, there has to be at least
  $y\in\V_{n+1}$ satisfying $d(x,y)<2\delta^{n+1}$.  Using
  $\delta<\frac{1}{6}$ this implies $d(x,y)<\frac{1}{3}\delta^n$, and
  by \eqref{epsseparated} we obtain $(x,y)\in\E$.  Hence $V$ can be
  represented by a subset $\T\subseteq\Sigma_N$ with $\V_n$
  corresponding to $\{x_I\}_{I\in \T^*}$ in the following way: If
  \begin{align*}
    (x^0,x^1_{i_1}),(x^1_{i_1},x^2_{i_2}),\dots,(x^{n-1}_{i_{n-1}},x^n_{i_n})\in\E
  \end{align*}
  is a path in $V$, then $I\coloneqq i_1 i_2\dots i_n \in \T_n$ and
  $x_I\coloneqq x^n_{i_n}$.  By \eqref{upperTnT} we have
  \begin{equation}\label{comparable}
    d (x_I,x_{IJ})\leq 2\sum\limits_{k=|I|}^{\infty}\delta^k=\frac{2}{1-\delta}\delta^{|I|}=:D_{\delta}\delta^{|I|}
  \end{equation}
  for every $I,J\in\T^*$, which shows that
  $\{x_{\omega|_n}\}_{n=0}^\infty$ is a Cauchy sequence. This allows
  us to define $x_\omega$ and $K_I$ for all $\omega\in\T$ and
  $I\in\T^*$ in the usual way. Furthermore,
  \begin{equation}\label{incomparable}
    \begin{aligned}
      d (x_I,x_{JiL})&\geq  d (x_I,x_{Ji})- d (x_{Ji},x_{JiL})\\
      &\geq \frac{\delta^{|I|}}{2}-\frac{2\delta^{|Ji|}}{1-\delta}
      \geq
      \left(\frac{1}{2}-\frac{2\delta}{1-\delta}\right)\delta^{|I|}=:2C_{\delta}\delta^{|I|}
    \end{aligned}
    \end{equation}
  when $I,JiL\in\T^*$, $|I|=|J|$, and $I$, $J$ are incomparable.  Combining
  \eqref{incomparable}, the fact that $ 2C_\delta<1 $ and \eqref{epsseparated}, we obtain
  \begin{equation*}
    d (x_I,x_J)\geq 2C_{\delta}\delta^{|I|}
  \end{equation*}
  provided $I,J\in\T^*$, $|I|\leq|J|$, $I$, $J$ incomparable.
  Next put
  \begin{equation*}
    r_I\coloneqq \mu\left(K_I\smallsetminus \bigcup_{\substack{|I|=|J|\\ J<_{\mathrm{lex}}I}} K_J\right)^{\frac{1}{s}}=:\mu(M_I)^{\frac{1}{s}},
  \end{equation*}
  where $<_{\mathrm{lex}}$ denotes the lexicographical order of finite
  words.  It is not difficult to see that for all $n\in\N$ and
  $I\in \T^*$,
  \begin{align*}
    \bigcup_{|J|=n}M_{IJ}=M_I,
  \end{align*}
  where the union on the left side is disjoint.  Indeed, the
  disjointness is clear from the definition, and one can write
  \begin{equation*}
    \begin{aligned}
      \bigcup_{|J|=n}M_{IJ}
      &=\bigcup_{|J|=n} \left(K_{IJ}\smallsetminus \bigcup_{\substack{|L|=|I|+n \\ L<_{\mathrm{lex}} IJ}}K_L\right)\\
      &=\left(\bigcup_{|J|=n} K_{IJ}\right)\smallsetminus \left(\bigcap_{|J|=n} \bigcup_{\substack{|L|=|I|+n \\ L<_{\mathrm{lex}} IJ}}K_L\right)\\
      &=K_{I}\smallsetminus \left(\bigcap_{|J|=n} \left(\bigcup_{\substack{|L|=|I| \\ L<_{\mathrm{lex}} I}}K_L\cup \bigcup_{\substack{|L|=n \\ L<_{\mathrm{lex}} J}}K_{IL}\right)\right)\\
      &=K_{I}\smallsetminus \left(\bigcup_{\substack{|L|=|I| \\
            L<_{\mathrm{lex}} I}}K_L\right)=M_{I}.
    \end{aligned}
  \end{equation*}
  From this, $r_I^s= \sum_{|J|=n}r_{IJ}^s$ for all $n\in\N$ and
  $I\in \T^*$ follows, which implies \ref{tri}.  By \eqref{comparable}
  and \eqref{incomparable} we have
  \begin{equation*}
   K\cap B(x_I,C_{\delta}\delta^{|I|})\subseteq M_I\subseteq K_I\subseteq B(x_I,D_{\delta}\delta^{|I|})
  \end{equation*}
  and therefore
  \begin{equation*}
  \begin{aligned}
  \alpha C_{\delta}^s(\delta^{|I|})^s&\leq \mu(B(x_I,C_{\delta}\delta^{|I|}))=\mu(K\cap B(x_I,C_{\delta}\delta^{|I|})) \\
  &\leq \mu(M_I)\leq \mu(B(x_I,D_{\delta}\delta^{|I|}))\leq \beta D_{\delta}^s(\delta^{|I|})^s
  \end{aligned}
  \end{equation*}
  which in particular implies
  \begin{equation}\label{equivalent}
    \tilde C_{\delta}\delta^{|I|}\leq r_I\leq\tilde D_{\delta}\delta^{|I|}
  \end{equation}
  for some $\tilde C_{\delta},\tilde D_{\delta}>0$.  Now, equations
  \eqref{comparable} and \eqref{incomparable} together with
  \eqref{equivalent} imply \ref{jedna} and \ref{dva} with appropriate
  $C$ and $D$
\end{proof}

We conclude the section with the following one simple but useful
observation concerning the $s$-trees.

\begin{lem}\label{prunnedTreeLemma}
  Let $(\T,\{x_I\}_{I\in\mathcal{T}^*},\{r_I\}_{I\in\mathcal{T}^*})$
  be an $s$-tree and suppose that \ref{tri} is replaced with \ref{triAlt}. Pick some $i(I)\in\Sigma_N$ for every $I\in\T^*$ in a way
  that $Ii(I)\in\T^*$.  Define for each $ I\in \T^* $
  \begin{equation*}
    \widetilde{\T}^*(I)\coloneqq\{J\in\T^*:\; IMi(IM)\nprec IJ\quad\text{for any}\quad M\in\T^*\}.
    \end{equation*}
  Then
  \begin{equation}\label{sumOfrIinTildaT}
    \sum_{\stackrel{J:J\in\widetilde{\T}^*(I)}{|J|=m}} r^s_{IJ}\leq E^{2}r_I^s\cdot(1-\rho^s)^{m}.
  \end{equation}
  for every $I\in\T^*$ and $m\in\en$.
 
\end{lem}
\begin{proof}
We will assume that condition \ref{tri} holds, the general case then follows from
Remark~\ref{r:triAltToTri}.  We will proceed by induction in $m$.
	
  Pick $I$. For $m=1$ formula \eqref{sumOfrIinTildaT} reduces to
  \begin{equation}\label{sumOfrIinTildaTm=1}
    \sum_{\stackrel{i:Ii\in\T^*}{i\not=i(I)}} r^s_{Ii}\leq r_I^s\cdot(1-\rho^s).
  \end{equation}
  To prove this we can write 
  \begin{equation*}
    1=\sum_{i:Ii\in\T^*} \left(\frac{r_{Ii}}{r_I}\right)^s=\sum_{\stackrel{i:Ii\in\T^*}{i\not=i(I)}}\left(\frac{r_{Ii}}{r_I}\right)^s
    +\left(\frac{r_{Ii(I)}}{r_I}\right)^s\geq \sum_{\stackrel{i:Ii\in\T^*}{i\not=i(I)}}\left(\frac{r_{Ii}}{r_I}\right)^s+\rho^s
  \end{equation*}
  which is the same as \eqref{sumOfrIinTildaTm=1}.
	
For the induction step, suppose that \eqref{sumOfrIinTildaT} holds up to $m-1$. Then
  \begin{equation*}
    \begin{aligned}
      \sum_{\stackrel{J:J\in\widetilde{\T}^*(I)}{|J|=m}} r^s_{IJ}
      =&\sum_{\stackrel{iJ:iJ\in\widetilde{\T}^*(I)}{|J|=m-1,}} r^s_{IiJ}=\sum_{\stackrel{i,J:J\in\widetilde{\T}^*(Ii)}{|J|=m-1, i\not=i(I)}} r^s_{IiJ}\\
      \leq&
      (1-\rho^s)^{m-1}\cdot\sum_{\stackrel{i:Ii\in\T^*}{i\not=i(I)}}
      r^s_{Ii}\leq r_I^s\cdot(1-\rho^s)^{m},
    \end{aligned}
  \end{equation*}
  which is what we want.
\end{proof}


\section{Measurability}\label{S:measurability}
\subsection{Setup and main results}\label{SS:setup}
In this whole section, 
$N$ will be a function satisfying \ref{C:monotonicity}-\ref{C:comparability}. 
Recall that those come with the constants $A$, $B$ and $G$ from the conditions \ref{C:separation}, \ref{C:lipschitzImage} and \ref{C:comparability}, respectively.

Let $(X,d)$ be a complete, doubling metric space and $K\subseteq X$ be an $s$-set.
Recall the conventions from Remark~\ref{convention} and let
$\mathrm{T}=(\T,\{x_I\}_{I\in\T^*},\{r_I\}_{I\in\T^*})$ be a corresponding
$s$-tree to $K$, with constants $C$, $D$, and $\rho$.
We will also assume \ref{triAlt} instead of \ref{tri} with the corresponding constant $E$.

Throughout this
section we impose several strong additional assumptions on $\mathrm{T}$, to
prove the main result, which concerns the asymptotic behaviour of the
function
\begin{align*}
  \eps\mapsto \eps^s N(\eps,K).
\end{align*}
 It is well known that both $\limsup$ and $\liminf$ of
this function are always positive and finite when $K$ is an $s$-set
(cf. Lemma~\ref{L:differentEpsilons}). The question whether the actual
limit exists is more delicate.  This section heavily uses the
techniques developed in \cite{Lalley1989}, especially the strategy of
the proof of the ``packing measurability'' for the limits sets of
Schottky groups.  The only essential difference is the fact that we
have to deal with overlaps, meaning that the unions
$K_I=\bigcup_{|J|=n}K_{IJ}$ are not necessarily disjoint
(cf. Lemma~\ref{L:estimateOverlaps}).  First we will assume that
\begin{enumerate}[label={\rm (M\thedefcounter)}]
  \refstepcounter{defcounter}\item\label{M:finiteType} $(\T,\sigma)$
  is a subshift of finite type,
  \refstepcounter{defcounter}\item\label{M:inclusionx_I}
  $\{x_I\}_{I\in\T^*} \subseteq K$,
  \refstepcounter{defcounter}\item\label{M:asymptotic}$r_{Ii}\leq
  Rr_{I}$ for some $0<R<1$.
\end{enumerate}

Note, that for $I,IJ\in\mathcal{T}^*$ it follows from \ref{M:inclusionx_I}, that
\begin{align}
\label{inclusion_x_IJ}
x_{IJ}\in K_I.
\end{align}
The case $I=\varnothing$ is clear, as
$K_\varnothing=K$.
Assume $\varnothing\neq I\in \mathcal{T}^*$.
As $x_{IJ}=x_\omega\in K=\bigcup_{|L|=|I|}K_L$ for some sequence
$\omega\in\mathcal{T}$, there is an $L\in\mathcal{T}_{|I|}$ such that
$x_{IJ}\in K_L$.
If $IJ$ is incomparable to $L$, then \ref{jedna} yields
$d(x_{IJ},x_{\omega|_n})\ge Cr_{IJ}$ for all $n\ge |L|$, which is a
contradiction to $x_\omega=\lim_{n\to\infty}x_{\omega|_n}$.
Thus $L\prec IJ$, which means $I=L$.

Observe that when imposing \ref{M:finiteType} and \ref{M:inclusionx_I} we, in particular, obtain the following:
there is a constant $F>0$ such that
\begin{equation}\label{eq:oppositeToT2}
\diam K_J\geq Fr_J
\end{equation}
for every $J\in\T^*$ (compare this with \eqref{diamKI}, which holds for any $s$-tree).
To prove this observation first note that since $(\T,\sigma)$ is a subshift of finite type (due to \ref{M:finiteType}),
there is some $p\in\en$ such that for every $J\in\T^*$ there are $I,\tilde I\in\T^*$, $I\not=\tilde I$ 
with $|I|=|\tilde I|=p$ and $JI,J\tilde I\in\T^*$. 
Now
\begin{equation*}
\diam K_J\stackrel{\eqref{inclusion_x_IJ}}{\geq} \dist(x_{JI},x_{J\tilde I})\stackrel{\ref{jedna}}{\geq} Cr_{JI}\stackrel{\ref{pet}}{\geq} C\rho^p r_J\eqqcolon Fr_J. 
\end{equation*}
In particular, this implies that there are
constants $\Upsilon_{m}>0$, $m\in\en$, such that for any
$I, J\in \mathcal{T}^*$
\begin{equation}\label{diam_estimate}
  \diam(K_{IJ})\geq \Upsilon_{|J|}\diam(K_{I}).
\end{equation}
Indeed, we can write
\begin{equation*}
\diam K_{IJ}\stackrel{\eqref{eq:oppositeToT2}}{\geq}  Fr_{IJ}\stackrel{\ref{pet}}{\geq}  F\rho^{|J|}r_I\stackrel{\ref{dva}}{\geq}  \frac{F\rho^{|J|}}{D}\diam K_I\eqqcolon \Upsilon_{|J|}\diam K_I.
\end{equation*}

Next we will suppose that 
\begin{measAssumption}\label{M:mappings}
	\begin{aligned}
	  \text{ there are bi-Lipschitz mappings $\varphi_i: X\to X$  and $W,\delta_0>0$ such that }\\ 
	  \varphi_{I}(K_J)=K_{IJ}\text{ and such that }
	   (K_I)_{r_IWG\delta}\subseteq \varphi_I(K_{G\delta}) \text{ for every } \delta\leq\delta_0.
	\end{aligned}
\end{measAssumption}
Here we used the usual notation
$\varphi_{I}\coloneqq\varphi_{I_{1}}\circ\dots\circ\varphi_{I_{|I|}}$ for
$I\neq \varnothing$, and $\varphi_\varnothing\coloneqq \mathrm{id}$.
In the applications we will usually assume that
$(\T,\sigma)=(\Sigma_{N},\sigma)$ is the full shift with $N$ symbols,
and the mappings $\varphi_{i}$ are then constructed as
$\varphi_{i}(x_{\omega})\coloneqq x_{i\omega}$, $\omega\in\T$.

Recall that $G$ is the constant guaranteed for the counting function $N$ by property \ref{C:lipschitzImage}.
Pick {\color{black} $\delta_0>\xi>0$} and let $\kappa^{\pm}_{i,J}=\kappa^{\pm}_{i,J}(G,\xi)$ be the optimal  constants such that
\begin{equation}\label{biLip1}
  \kappa^{-}_{i,J}\frac{r_{iJ}}{r_J} d(x,y)
  \leq d(\varphi_i(x),\varphi_i(y))
  \leq\kappa^{+}_{i,J}\frac{r_{iJ}}{r_J} d(x,y)
\end{equation} 
whenever $x,y\in \varphi_J(K_{G\xi})$. 

Note that by iterating \eqref{biLip1} one obtains
\begin{equation}\label{biLip2}
  \kappa^{-}_{I,J}\frac{r_{IJ}}{r_J} d(x,y)
  \leq d(\varphi_I(x),\varphi_I(y))
  \leq\kappa^{+}_{I,J}\frac{r_{IJ}}{r_J} d(x,y)
\end{equation}
whenever $x,y\in \varphi_J(K_{G\xi})$.
Here we denoted
\begin{equation}\label{kappaI}
  \kappa^{\pm}_{I,J}\coloneqq\prod\limits_{n=1}^{|I|}\kappa^{\pm}_{I_{n},\sigma^n(I)J}.
\end{equation}

In what follows we will assume that 
\begin{measAssumption}\label{M:kappa_to_one}
	\kappa^{\pm}_{I,J}\to 1\quad\text{and}\quad  \quad\text{as}\quad |J|\to\infty\quad \text{uniformly in $I$}.
\end{measAssumption}

Now put
\begin{align*}
  f_1(iJ):=\log\left(\frac{r_{J}}{r_{iJ}\kappa^{+}_{i,J}}\right)
  \quad\text{and}\quad 
  f_2(iJ)\coloneqq \log\left(\frac{r_{J}}{r_{iJ}\kappa^{-}_{i,J}}\right),
\end{align*}
which yields for $ IJ\in\T^* $, $I\neq \varnothing$,
\begin{equation}\label{422342552343236}
  f_1(IJ)=\log\left(\frac{r_{\sigma(I)J}}{r_{IJ}\kappa^{+}_{I_1,\sigma(I)J}}\right)
  \text{  and  }
  f_2(IJ)=\log\left(\frac{r_{\sigma(I)J}}{r_{IJ}\kappa^{-}_{I_1,\sigma(I)J}}\right).
\end{equation}
Set $f_1(\varnothing)\coloneqq f_2(\varnothing)\coloneqq 0$.  
Also, by \eqref{kappaI} we have for
all  $IJ\in\T^*$, $I\neq \varnothing$, (recall the definition of $S_nf$ in \eqref{eq:ergodicSum})
\begin{equation}\label{SIf}
  S_{|I|}f_{1}(IJ)=\log\left(\frac{r_{J}}{r_{IJ}\kappa^{+}_{I,J}}\right)
  \quad\text{and}\quad 
  S_{|I|}f_{2}(IJ)=\log\left(\frac{r_{J}}{r_{IJ}\kappa^{-}_{I,J}}\right).
\end{equation}

Assume additionally
\begin{measAssumption}\label{M:rho_Jcontinuous}
	\begin{aligned}
		&\frac{r_{i\omega|_n}}{r_{\omega|_n}}\to\rho_{i\omega}\in(0,1)\quad\text{as}\quad n\to\infty,
		\text{ and $\omega\mapsto \rho_{\omega}$ continuous on $\T$ and, }\\
		&\text{moreover, there are H\"older continuous
		extensions $\tilde f_i$ of $f_i$ to $\T\cup\T^*$}\\ &\text{such that
		$\tilde f_1(\omega)=\tilde f_2(\omega)=\log\left(\frac{1}{\rho_{\omega}}\right)\eqqcolon f(\omega)$.}
	\end{aligned} 
\end{measAssumption}
Note that $f$ is H\"older continuous and $f>0$.

Finally assume
\begin{measAssumption}\label{M:holder/nonlattice}
  \text{$f$ is a non-lattice function on $\T$} .
\end{measAssumption}

If one has an $s$-tree which satisfies above conditions, by switching to
a higher power of the underlying dynamics, one obtains another $s$-tree which generates
the same $s$-set.
The advantage of the new tree are smaller ratios, which is of importance in the later proofs.  
Thus before we proceed, we state the following technical lemma:
\begin{lem}\label{relabel_trick}
Let $\mathrm{T}$ be an $s$-tree and $K$ its generated $s$-set.
Suppose $\mathrm{T}$ satisfies \ref{M:finiteType}-\eqref{M:holder/nonlattice}.
Then for each $\eps >0$ there is an $s$-tree $\mathrm{T}'$, such that
\begin{itemize}
\item[(1)] $\mathrm{T}'$ satisfies \ref{M:finiteType}-\eqref{M:holder/nonlattice};
\item[(2)] $\rho',r'_{I'},\rho'_{\omega'}, R'<\eps $ for  $I'\in(\T')^*,\;\omega'\in\T'$;
\item[(3)] $\mathrm{T}'$ generates $K$ as well.
\end{itemize}
\end{lem}
\begin{proof}
For given $m\ge 1$, the set $\T_m$ of words with length $m$ in $\T^*$ can be seen as as alphabet of the address space $\Sigma_{N_m}$,
where $N_m\coloneqq \#\T_m$.
Define the transition matrix $A'\in\{0,1\}^{N_m\times N_m}$ such that
\begin{align*}
A'_{I_1\dots I_m ,J_1\dots J_m}=1 \quad\iff\quad A_{I_m, J_1}=1.
\end{align*}
In this way we get a new subshift of finite type $(\T',\sigma')$, where
\begin{align*}
\T'\coloneqq\left\{\omega\in\Sigma_{N_m}: A'_{\omega_{n},\omega_{n+1}}=1\quad\text{for all $n\in\en$} \right\}.
\end{align*}
Note, that  $(\T',\sigma')$ is topologically conjugated to  $(\T,\sigma^m)$, where $\sigma^m\coloneqq \sigma\circ\dots\circ\sigma$
is the $m$-th power of $\sigma$. In addition, the following mapping gives rise to a homeomorphism $\Phi:\T\to\T'$:
\begin{align*}
\T\ni\omega=\omega_1\omega_2\dots\omega_m\omega_{m+1}\dots\omega_{2m}\dots \mapsto (\omega_1\omega_2\dots\omega_m)(\omega_{m+1}\dots\omega_{2m})\dots=\omega'\in\T'.
\end{align*}
Using this relation, define for each $I'=(I^1_1\dots I^1_m)\dots(I^k_1\dots I^k_m)\in\T'_k$ and $\omega'\in\T'$
\begin{align*}
x'_{I'}&\coloneqq x_{I^1_1\dots I^1_m \dots I^k_1\dots I^k_m},\\
r'_{I'}&\coloneqq r_{I^1_1\dots I^1_m \dots I^k_1\dots I^k_m},\\
\varphi'_{I'}&\coloneqq \varphi_{I^1_1\dots I^1_m \dots I^k_1\dots I^k_m}\\
x'_{\omega'}&\coloneqq \lim_{n\to\infty}x'_{\omega'|_n}=\lim_{n\to\infty}x_{\omega|_{mn}}.
\end{align*}
Note that for all $I'\in\T'$ one has $r'_{I'}\le R^m$.
Clearly \ref{jedna}-\ref{pet} are satisfied with $C'\coloneqq C$, $D'\coloneqq D$ and $\rho'\coloneqq \rho^m$.
In particular \ref{triAlt} holds with $E'\coloneqq E$.
It is also immediate that
\begin{align*}
K=\{x_{\omega}:\omega\in\T, x_{\omega}\;\text{exists}\}\subseteq \{x'_{\omega'}:\omega'\in\T', x'_{\omega'}\;\text{exists}\}=K'.
\end{align*}
As $(X,d)$ is complete, $x_{\omega}$ does exist for each $\omega\in\T$, thus $K=K'$, which shows \textit{(3)}.
This also gives \ref{M:inclusionx_I}.
The property \ref{M:asymptotic} is satisfied by $R'\coloneqq R^m$.
The property \eqref{M:mappings} is trivially satisfied using $\varphi'_{I'}$.
For \eqref{M:kappa_to_one} note, that for all $I'=(I^1_1\dots I^1_m) \dots (I^l_1\dots I^l_m), J'=(J^1_1\dots J^1_m) \dots (J^k_1\dots J^k_m)\in(\T')^*$ one has by (\ref{M:kappa_to_one})
\begin{align*}
\kappa'^{\pm}_{I',J'}\coloneqq\kappa^{\pm}_{I, J}\to 1\  \quad\text{as}\quad |J'|\to\infty,
\end{align*}
where $I\coloneqq I^1_1\dots I^1_m \dots I^l_1\dots I^l_m$ and $J\coloneqq J^1_1\dots J^1_m \dots J^k_1\dots J^k_m$.
As $\omega\mapsto\rho_\omega\in(0,1)$ is continuous on a compact space, there is an $L<1$ such that $\rho_\omega\le L$ for all $\omega\in\T$.
This yields in a similar way as above 
\begin{align}\label{prod_struct}
\frac{r'_{i'\omega'|_n}}{r'_{\omega'|_n}}\to\prod_{n=0}^{m-1}\rho_{\sigma^n(i_1\dots i_m)\omega}\eqqcolon \rho'_{i'\omega'}\in(0,L^m)\subseteq (0,1)\quad\text{as}\quad n\to\infty
\end{align}
for all $i'=i_1\dots i_m\in\T'_1$ and $\omega'\in\T'$.
Note, that as $\omega\mapsto \rho_\omega$ and $\sigma:\T\to\T$ are continuous, also $\omega\mapsto \rho_{\sigma^n(\omega)}$ is continuous for each $n\ge 0$.
The continuity of $\omega'\mapsto \rho'_{\omega'}=\prod_{n=0}^{m-1}\rho_{\sigma^n(\omega)}$ follows then from the continuity of $\omega'\mapsto \omega$.

For \eqref{M:holder/nonlattice} first note, that if $f,g:\T\to\R$ are $\alpha$-H\"older continuous, then also $f+g$ and $f\circ \sigma^n$ are  $\alpha$-H\"older continuous for
each $n\ge 0$.
Next, observe that for each $\omega\in\T'$
\begin{align*}
f'(\omega')\coloneqq\log\left(\frac{1}{\rho'_{\omega'}}\right)\stackrel{\eqref{prod_struct}}{=}\sum_{n=0}^{m-1}\log\left(\frac{1}{\rho_{\sigma^n(\omega)}}\right)
=\sum_{n=0}^{m-1}f(\sigma^n(\omega))>0,
\end{align*}
where $f:\T\to\R$ is the non-lattice $\alpha$-H\"older continuous extension of $f_1$ and $f_2$.
Using this we obtain for each $k\ge 0$
\begin{align*}
\frac{\var_k(f')}{\alpha^k}\le \alpha^{m}\sum_{n=0}^{m-1}\frac{\var_{mk}(f\circ \sigma^n)}{\alpha^{mk}},
\end{align*}
thus $|f'|_\alpha\le \alpha^{m}\sum_{n=0}^{m-1}|f\circ \sigma^n|_\alpha <\infty$.
Hence $f'$ is $\alpha$-H\"older on $(\T,\sigma^m)$.
To show that $f'$ is non-lattice on $(\T,\sigma^m)$, one can proceed like in the proof of Lemma 13.1 in \cite{Lalley1989}.
For convenience of the reader, we present the following elementary proof.

We would like to prove that if $f'$ is lattice then $f$ is lattice as well.
Suppose that $f'$ is lattice.
This means that there is a discrete subgroup $G'\subseteq\er$ and $g':\T'\to G'$ such that $f'$ is cohomologous to $g'$.
Define 
\begin{equation*}
g(\omega)=\frac1m g'(\omega')
\end{equation*}
and put $G\coloneqq \frac1m G'$.
Clearly $g:\T\to G$ and $G$ is a discrete subgroup of $\er$.
We need to prove that $f$ is cohomologous to $g$, which will be a contradiction with the fact that $f$ is non-lattice.

Choose $p\in\en$ and suppose that $\omega\in\T$ such that $\sigma^p\omega=\omega$.
Note that $\sigma^p\sigma^k\omega=\sigma^k\omega$ holds for any $k\in\en$.
Then

\begin{equation}\label{eq:SpisSpSigma}
S_p f(\sigma^k\omega)=S_pf(\omega)\quad \text{for every}
\quad k\in\en
\end{equation}
which in particular implies
\begin{equation}\label{eq:SpisSmp}
S_pf(\sigma^k\omega)=\frac1m S_{pm}f(\sigma^k\omega)
\quad \text{for every}\quad k\in\en_0.
\end{equation}
We also have
\begin{equation}\label{eq:gis1overmg}
g(\sigma^k\omega)=\frac{1}{m^2}\sum_{n=0}^{m-1}g'((\sigma^{np}\sigma^k\omega)')
\quad \text{for every}\quad k\in\en_0
\end{equation}
and
\begin{equation}\label{eq:omega'periodic}
(\sigma')^p(\sigma^k\omega)'=(\sigma^{pm+k}\omega)'=(\sigma^k\omega)'
\quad \text{for every}\quad k\in\en_0
\end{equation}

Then 
\begin{equation*}
\begin{aligned}
S_pf(\omega)&\stackrel{\eqref{eq:SpisSpSigma}}{=}
\frac1m \sum_{k=0}^{m-1}S_pf(\sigma^k\omega)
\stackrel{\eqref{eq:SpisSmp}}{=}
\frac{1}{m^2}\sum_{k=0}^{m-1} S_{pm}f(\sigma^k\omega)\\
&=\frac{1}{m^2}\sum_{k=0}^{m-1}\sum_{n=0}^{pm-1}
f(\sigma^{n}\sigma^k\omega)
=\frac{1}{m^2}\sum_{k=0}^{m-1}\sum_{j=0}^{p-1}\sum_{l=0}^{m-1}
f(\sigma^{l}\sigma^{jm}\sigma^k\omega)\\
&=\frac{1}{m^2}\sum_{k=0}^{m-1}\sum_{j=0}^{p-1}
f'((\sigma^{jm}\sigma^k\omega)')
=\frac{1}{m^2}\sum_{k=0}^{m-1}\sum_{j=0}^{p-1}
f'((\sigma')^j(\sigma^k\omega)')\\
&=\frac{1}{m^2} \sum_{k=0}^{m-1} S_p f'((\sigma^k\omega)')
\stackrel{\eqref{eq:omega'periodic}}{=}
\frac{1}{m^2} \sum_{k=0}^{m-1} S_p g'((\sigma^k\omega)')\\
&=\frac{1}{m^2} \sum_{k=0}^{m-1}\sum_{j=0}^{p-1}
g'((\sigma')^j(\sigma^k\omega)')
=\frac{1}{m^2} \sum_{k=0}^{m-1}\sum_{j=0}^{p-1}
g'((\sigma^{jm+k}\omega)')\\
&=\frac{1}{m^2} \sum_{k=0}^{pm-1} g'((\sigma^k\omega)')
=\frac{1}{m^2} \sum_{k=0}^{p-1}\sum_{j=0}^{m-1}
g'((\sigma^{jp+k}\omega)')\\
&=\sum_{k=0}^{p-1}\frac{1}{m^2} \sum_{j=0}^{m-1}
g'((\sigma^{jp}(\sigma^k\omega))')
\stackrel{\eqref{eq:gis1overmg}}{=}\sum_{k=0}^{p-1}g(\sigma^k\omega)
=S_p g(\omega)
\end{aligned}
\end{equation*}
which proves that $f$ is cohomologous to $g$.

Now for $i'=i_1\dots i_m\in\T'_1$ and $J'\in(\T')^*$, by definition one has
\begin{align*}
f'_i(i'J')&=\log\left(\frac{r'_{J'}}{r'_{i'J'}\kappa'^{\pm}_{i',J'}}\right)=
\sum_{n=1}^m\log\left(\frac{r_{\sigma^n(i_1\dots i_m)J}}{r_{\sigma^{n-1}(i_1\dots i_m)J}}\kappa^{\pm}_{i_n,\sigma^n(i_1\dots i_m)J}\right)\\
&=\sum_{n=0}^{m-1} f_i(\sigma^n(i_1\dots i_m)J).
\end{align*}
This yields in the same way as above, that there are $\alpha$-H\"older continuous extensions $\tilde{f}'_i$ to $(\T'\cup(\T')^*,\sigma')$ with $\tilde{f}'_i(\omega')=f'(\omega')$ for all $\omega'\in \T'$. Thus \textit{(1)} holds.

Finally, as $m\ge 0$ can be chosen arbitrarily large, \textit{(2)} follows.
\end{proof}

{\color{black}
As a first application of the above lemma, we show that we may always assume  
\begin{equation}\label{eq:K_GxiInclusion}
\varphi_{IJ}(K_{G\xi})\subseteq \varphi_{I}(K_{G\xi})
\end{equation}
for any $I,J\in\T^*$.
To prove this we just need 
\begin{equation}
\varphi_{J}(K_{G\xi})\subseteq K_{G\xi}
\end{equation}
for any $J\in\T^*$.
But since $\varphi_{J}$ is $\kappa^+_{J,\varnothing} r_J$-Lipschitz on $K_{G\xi}$, we in fact just need $\kappa^+_{J,\varnothing}  r_J\leq 1$ for any $J\in\T^*$, 
which is always possible due to \eqref{M:kappa_to_one} and Lemma~\ref{relabel_trick}.}

Now, \eqref{M:kappa_to_one} in particular implies that there is a constant 
{\color{black}$\frac{1}{W}<\kappa<\infty$ (here $W>0$ is from condition \eqref{M:mappings})} such that 
\begin{equation}\label{M:kappa}
\frac{1}{\kappa}\leq \kappa^{\pm}_{I,J}\leq \kappa
\end{equation}
for every $I,J\in\mathcal{T}^*$.
This yields
\begin{align}\label{rough_4.5_estimate}
\frac{r_{IJ}}{\kappa r_J} d(x,y)\stackrel{\eqref{M:kappa}}{\leq} \kappa^{-}_{I,J}\frac{r_{IJ}}{r_J} d(x,y)
\stackrel{\eqref{biLip2}}{\leq} d(\varphi_{I}(x),\varphi_{I}(y))
\end{align}
for  $x,y\in \varphi_{J}(K_{G\xi})$.
In particular, applying \eqref{rough_4.5_estimate} to $I\coloneqq J$ and $J\coloneqq \varnothing$, one has
\[
\frac{r_{J}}{\kappa  } d(x,y)\leq d(\varphi_{J}(x),\varphi_{J}(y)) 
\]
for  $x,y\in K_{G\xi}$.
Thus {\color{black}
\begin{align}\label{neighborhood_inclusion}
(K_J)_{\frac{G\xi r_J}{\kappa  }}\stackrel{ \eqref{M:mappings}} {\subseteq} \varphi_J(K_{G\xi}).
\end{align}}
That means that $ \varphi_I $ is $ \kappa^{+}_{I,J}\frac{r_{IJ}}{r_J} $-Lipschitz on $ (K_J)_{\frac{G\xi r_J}{\kappa  }} $.
Therefore, if $0<\eps \le \frac{\xi r_J}{\kappa  }$, we can apply \ref{C:lipschitzImage} and obtain
\begin{align}\label{upper_estimate_part_counting_function}
N\left(\kappa^{+}_{I,J}\frac{r_{IJ}}{r_J}\eps,K_{IJ}\right)
\leq N(\eps,K_J).
\end{align}
Next, given some $0<\eps ,L< \infty$, define $L\tilde{\eps }\coloneqq \eps $.
Assume  $\varphi_I^{-1}$ to be $L$-Lipschitz on $(K_{IJ})_{G\tau}$, and $\tilde{\eps }\le \tau$.
We have then
\begin{align}\label{upper_estimate_for_N}
N(\eps ,K_J)\stackrel{\eqref{M:mappings}}=N(L\tilde{\eps },\varphi_I^{-1}(K_{IJ}))\stackrel{\ref{C:lipschitzImage}}{\le} N(\tilde{\eps }, K_{IJ}).
\end{align}
By \eqref{neighborhood_inclusion} it follows that
\begin{align*}
(K_{IJ})_{\frac{G\xi r_{IJ}}{\kappa  }}\subseteq\varphi_{IJ}((K)_{G\xi})= \varphi_I(\varphi_J(K_{G\xi})).
\end{align*}
In addition, $\varphi_I^{-1}$ is $\left(  \kappa^{-}_{I,J}\frac{r_{IJ}}{r_J}\right)^{-1}$-Lipschitz on $\varphi_I(\varphi_J(K_{G\xi}))$.
Thus \eqref{upper_estimate_for_N} holds for $L=\left(  \kappa^{-}_{I,J}\frac{r_{IJ}}{r_J}\right)^{-1}$, if
$
\eps  L^{-1}=\tilde{\eps }\leq\frac{\xi r_{IJ}}{\kappa  }.
$
This is always satisfied for $\eps \leq\frac{\xi r_J}{\kappa^2  }$.
Combining \eqref{upper_estimate_part_counting_function} and \eqref{upper_estimate_for_N}, that means
\begin{equation}\label{packingMI}
N\left(\kappa^{+}_{I,J}\frac{r_{IJ}}{r_J}\eps,K_{IJ}\right)
\leq N(\eps,K_J)
\leq N\left(\kappa^{-}_{I,J}\frac{r_{IJ}}{r_J}\eps,K_{IJ}\right)
\end{equation}
for all $I,J\in\mathcal{T}^*$,
whenever
\begin{align}\label{eps_tilde}
\eps\leq \min\left(\frac{\xi\rho^{|J|}}{\kappa },\frac{\xi\rho^{|J|}}{\kappa^2 }\right)=\frac{\xi\rho^{|J|}}{\kappa^2 }
\eqqcolon \tilde\eps_{|J|}.
\end{align}

Note, that from \eqref{biLip2} also
\begin{equation}\label{diameterMI}
\kappa^{-}_{I,J}\frac{r_{IJ}}{r_J}\diam(K_{J})
\leq\diam(K_{IJ})
\leq\kappa^{+}_{I,J}\frac{r_{IJ}}{r_J}\diam(K_{J})
\end{equation}
for all $I,J\in\mathcal{T}^*$ follows.

Also, for all $ IJ\in \T^*$, equations \eqref{packingMI} and \eqref{diameterMI} can be
rewritten as
\begin{equation}\label{packingMIrev}
N\left(e^{-S_{|I|}f_{1}(IJ)}\eps,K_{IJ}\right)
\leq N(\eps,K_J)
\leq N\left(e^{-S_{|I|}f_{2}(IJ)}\eps,K_{IJ}\right)
\end{equation}
whenever $\eps\leq\tilde\eps_{|J|}$,
and
\begin{equation}\label{diameterMIrev}
e^{-S_{|I|}f_{2}(IJ)}\diam(K_{J})
\leq\diam(K_{IJ})
\leq e^{-S_{|I|}f_{1}(IJ)}\diam(K_{J}).
\end{equation}
Note, that \eqref{packingMIrev} and \eqref{diameterMIrev} are trivially satisfied, if $ I=\varnothing $.

{\color{black}
Using
\eqref{422342552343236}, \eqref{eq:K_GxiInclusion} and the optimality of the
constants $\kappa_{i,J}^{\pm}$, we immediately obtain
\begin{equation}\label{monotonefi}
f_2(I)\ge f_2(IJ)\ge f_1(IJ)\geq f_1(I)
\end{equation}
for every $IJ\in\T^*$, $I\neq \varnothing$.}

Now, condition \eqref{monotonefi} in particular implies that there is some $m_0\in\en$ such that
\begin{equation}\label{strictf}
f_1(iJ)>0 \quad\text{whenever $|J|\geq m_0$}.
\end{equation}
Indeed, assume there exists for each $n\ge 1$ an $I^n\in\bigcup_{j\ge n}\T_j$ such that
$f_1(iI^n)\le 0$.
Next choose some $I^n\prec \omega^n\in\T$.
As $\T$ is compact, there exists an $\omega\in\T$ and an increasing subsequence $(n_k)_{k\ge 1}$ such that 
$\omega^{n_k}\to \omega$ as $k\to\infty$.
Thus for each $l\in\N$ there is an $N_l\ge 1$ such that $\omega|_{l}=\omega^{n_k}|_{l}$
and $n_k\ge l$ for all $k\ge N_l$.
Using $n_k\le |I^{n_k}|$, this implies for all $l\ge 1$ and $k\ge N_l$
\begin{align*}
\omega|_{l}\prec I^{n_k},\quad \textrm{hence}\quad \tilde f_1(i\omega|_l)=f_1(i\omega|_l)\stackrel{\eqref{monotonefi}}{\le}  f_1(i I^{n_k})\le 0.
\end{align*}
As $\tilde f_1$ is continuous on $\T\cup \T^*$, it follows that
\begin{align*}
\log\left(\frac{1}{\rho_{i\omega}}\right)=\tilde f_1(i\omega)=\lim_{l\to\infty}\tilde f_1(i\omega|_l)\le 0,
\end{align*}
which is a contradiction to \eqref{M:rho_Jcontinuous}.

As a next application,
we show that we can always assume an existence of a function $\phi:\T^*\to\en_0$ such that
\begin{equation}\label{copy_in_ball}
  I\phi(I)\in\T^* \quad \text{and}\quad K_{I\phi(I)}\subseteq B\left(x_I,\frac{Cr_I}{2}\right)
\end{equation}
for every $I\in\T^*$.
Indeed, applying (\ref{inclusion_x_IJ}) and \ref{jedna}, we have
\begin{align*}
  B\left(x_I,\frac{C}{2}r_I\right)\cap K =  B\left(x_I,\frac{C}{2}r_I\right)\cap K_I.
\end{align*}
Now pick some $m\in\N$ such that $R^m  D\le\frac{C}{2}$, and some
$J=J(I)\in \mathcal{T}_m$ such that $x_I\in K_{IJ}$.  Using
\ref{M:asymptotic} and \eqref{diamKI} this implies that
$\mathrm{diam}(K_{IJ})\le\frac{C}{2}r_I$ and thus
\begin{equation*}
  K_{IJ}\subseteq B\left(x_I,\frac{Cr_I}{2}\right).
\end{equation*}
According to Lemma \ref{relabel_trick}, without loss of
generality we may assume that $m=1$, which shows (\ref{copy_in_ball}).

We now state the main results of this section:
\begin{prop}\label{bowenformula}
  Let $\mathrm{T}$ be an $s$-tree satisfying conditions
  \ref{M:finiteType}-\eqref{M:holder/nonlattice}. Then $s$ is the unique zero of the pressure function $t\mapsto p(-tf)$ (for the definition and properties of $p$ see \cite[Section~2B]{Bowen1975}),
where $f$ is defined as in \eqref{M:holder/nonlattice}.
\end{prop}
\begin{thm}\label{MainTheoren}
  Let $\mathrm{T}$ be an $s$-tree satisfying conditions
  \ref{M:finiteType}-\eqref{M:holder/nonlattice}. Then there exists
  $0<\theta<\infty$ such that
  \begin{equation*}
    \lim_{\eps\to 0+} \eps^s N(\eps,K)=\theta.
  \end{equation*}
\end{thm}
\begin{rem}
Before we proceed with the proofs of both statements,
 we will first explain what the conditions \ref{M:finiteType}-\eqref{M:holder/nonlattice} mean in the classical case of $K$ being a self-similar set.
This will not only explain their meaning, but will be also useful in Section~\ref{S:examples}.

Let $K\subseteq\er^d$ be a self-similar set generated by similarities
$\varphi_1,\dots,\varphi_N$ with corresponding contraction ratios
$r_1,\dots,r_N$ and similarity dimension $s$.  Suppose that the OSC
holds and that the set in non-lattice, namely there exist $i,j$ such
that $\frac{\log(r_i)}{\log(r_j)}\not\in\qe$.

Let $\mathrm{T}=(\Sigma_N,\{x_I\},\{r_I\})$ be some $s$-tree constructed in the
sense on Remark~\ref*{R:SSStoStree} with the additional property that
$\{x_I\}_{I\in\T^*} \subseteq K$.  In this case conditions
\ref{M:finiteType}-\ref{M:asymptotic} are clearly satisfied with
$R\coloneqq \max_i r_i$.  For conditions
\eqref{M:mappings}-\eqref{M:rho_Jcontinuous} it is sufficient to consider
\begin{equation*}
  \varphi_I\coloneqq\varphi_{I_1}\circ\dots\circ\varphi_{I_{|I|}},\quad\text{and}\quad\kappa_{i,J}^{\pm}\coloneqq\kappa\coloneqq 1.
\end{equation*}
The mapping $f$ is then defined by $f(\omega)\coloneqq-\log(r_{\omega_1})$ and we
only need to check condition \eqref{M:holder/nonlattice}.  The
H\"olderness of the mapping $f$ is immediate, since $f(\omega)$
depends only on the first coordinate of $\omega$ and it remains to
verify the non-lattice condition.

{\color{black}  To do
	this pick $i,j$ such that $\frac{\log(r_j)}{\log(r_i)}=r\not\in\qe$. and suppose that there is a function $g:\T\to\er$ such that
	\begin{equation}\label{eq:comohologous}
	f-g=u-u\circ\sigma
	\end{equation}
	for some function $u$ on $\T$ and such that $g$ takes values in a proper closed (additive) subgroup of $\er$ (i.e. there is $a>0$ such that $g(\omega)\in a\zet$ for every $\omega\in\T$).

Consider $\omega,\tau\in\T$ of the form $\omega\coloneqq ii\dots$, $\tau\coloneqq jj\dots$ and observe that \eqref{eq:comohologous} implies $
g(\omega)=-\log(r_i)$ and $g(\tau)=-\log(r_j)$.
Hence $\log(r_i),\log(r_j)\in a \zet$, which is a contradiction with the choice on $i$ and $j$.}
\end{rem}

\subsection{Proofs of the main results}
\begin{proof}[Proof of Proposition \ref{bowenformula}]
  Let $0<\delta<\infty$ be the unique zero of $t\mapsto p(-tf)$.
  As $-\delta f$ is $ \alpha $-H\"older continuous on $\mathcal{T}$, there exists a $\sigma$-invariant Gibbs measure $\nu$ with respect to $-\delta f$ on $\mathcal{T}$.
This means there is a constant $0<c<\infty$ such that for each admissible word $I\in\mathcal{T}^*$ and all $\omega\in [I]$
one has
\begin{align}\label{gibbsproperty}
\frac{1}{c}\le\frac{\nu([I])}{\exp\big(-\delta S_{|I|}f(\omega)\big)}\le c.
\end{align}
Our aim is to show that $K$ is $\delta$-regular with respect to $\nu\circ \pi^{-1}$, which implies $\delta=s$.

As $f=\tilde{f}_2$ on $\mathcal{T}$, we can pick by $ \alpha $-H\"older continuity of $\tilde{f}_2$ some constant $0\le S<\infty$ satisfying 
\begin{align}\label{holderestimatestuff}
\left| S_{|I|}f(\omega) -  S_{|I|}f_2(\omega|_{|I|})\right|\le|I|\cdot\mathrm{var}_{|I|}(\tilde{f}_2)\le S
\end{align}
for all $I\in\mathcal{T}^*$ and $\omega\in[I]$.
Using \eqref{gibbsproperty} together with \eqref{holderestimatestuff}, \eqref{diameterMIrev} and \eqref{diamKI}, one obtains
\begin{align*}
\nu([I])\le \frac{c  \exp(\delta S)D^\delta}{(\diam K)^\delta} r_I^\delta.
\end{align*}
Set $\mu\coloneqq \nu\circ\pi^{-1}$.
Repeating the proof of the upper bound of Lemma \ref{L:sTreeImpliesRegular},
one derives for all $x\in K$ and $1>r>0$
\[
\mu (B(x,r))\le \beta r^\delta,
\]
where $0<\beta<\infty$ is some constant.

For the lower estimate, recall that as $(\mathcal{T}, \sigma)$ is a subshift of finite type with transition matrix  $A$, there exists a $l_0\in\N$ such that $A^l$ is strictly positive for all $l\ge l_0$.
Thus, for fixed $I\in\mathcal{T}^*$, there are $J=J_1\dots J_{l_0}, J'=J_1'\dots J_{l_0}'\in\mathcal{T}_{l_0}$ such that $J_{l_0}\neq J_{l_0}'$ and $IJ, IJ'\in\mathcal{T}^*$.
This means by \ref{dva} and \ref{pet}
\[
d(x_{IJ},x_{IJ'})\stackrel{\ref{jedna}}{\ge} C (r_{IJ_1\dots J_{l_0}}+r_{IJ_1'\dots J_{l_0}'})\stackrel{\ref{pet}}{\ge} 2C\rho^{l_0} r_{I}.
\]
As $K_{I}=\bigcup_{|J|=l_0}K_{IJ}$, above estimate yields $\diam K_I\ge 2C\rho^{l_0} r_{I}$.
With similar arguments like before, we obtain analogously to \eqref{eq:lower} the estimate
\[
\mu(K_I)\geq\nu([I])\ge\frac{\exp(-\delta S)}{c}\Big(\frac{2C\rho^{l_0}}{\diam K}\Big)^{\delta} r_I^{\delta}\eqqcolon \alpha'r_I^{\delta}.
\]
Now repeating the proof of the lower bound in Lemma \ref{L:sTreeImpliesRegular},
one derives
\[
\frac{\mu (B(x,r))}{r^{\delta}}\ge  \frac{\alpha'\rho^{\delta}}{D^{\delta}}\eqqcolon \alpha
\]
for all $x\in K$ and $\diam K>r>0$.
\end{proof}

\begin{rem}
In the proof of Proposition \ref{bowenformula}, we did not use \ref{tri} or \ref{triAlt}.
Thus, a tree which satisfies \ref{jedna}, \ref{dva}, \ref{ctyri}, \ref{pet} and
\ref{M:finiteType}-\eqref{M:holder/nonlattice}, generates a $\delta$-regular set,
where $\delta$ is the root of the pressure function.
This also holds, if one drops the non-lattice condition in \eqref{M:holder/nonlattice},
as it is not needed for the existence of the Gibbs measure.
Note however, that the underlying space $ (X,d) $ needs to be complete and doubling.
\end{rem}
To prove Theorem \ref{MainTheoren}, we need to do some preparations.  Fix
$\eps>0$. Pick for $m\ge 1$ some positive constants $\gamma_m$ (their
actual value will be determined later - see
\eqref{how_to_choose_gamma_m}).

Recall $m_0$ to be the constant determined by \eqref{strictf}.
Next consider $I\in\mathcal{T}^*$
and $J\in\bigcup_{k>m_0}^\infty\mathcal{T}_k$. Then $|IJ|>m_0$, and by
\eqref{strictf} we have
\begin{equation}\label{24590760w6459}
  S_{k}f_{1}(IJ)> 0
\end{equation}
for every $k=0,\dots, |I|$. Furthermore consider the conditions
\begin{equation}\label{lowergammam}
  S_{|I|}f_{1}(IJ)>-\log(\eps)-\gamma_{ |J| },
\end{equation}
\begin{equation}\label{uppergammam}
  S_{k}f_{1}(IJ|_{ |J| +k})\leq-\log(\eps)-\gamma_{ |J| },\; k=0,\dots, |I|-1,
\end{equation}
\begin{equation}\label{uppergammamalt}
  S_{|I|-1}f_{1}(IJ)\leq-\log(\eps)-\gamma_{ |J| }.
\end{equation}
Note that in the case $I=\varnothing$ condition \eqref{uppergammam} is always satisfied. Also, in the case $|I|\leq1$, we have that $ S_{|I|-1}f_{1}(IJ)=0$.
We define for $m\ge 1$
\begin{align*}
  \alpha_{m}\coloneqq \sup_{IJ: |J| =m}\big(S_{|I|}f_{2}(IJ)-S_{|I|}f_{1}(IJ)\big).
\end{align*}
Note that \eqref{M:kappa} immediately implies
\begin{align*}
  \log\kappa^{-2}\le\alpha_m\le \log\kappa^2  
\end{align*}
for all $m\ge 1$ and since
\begin{align*}
  S_{|I|}f_{2}(IJ)-S_{|I|}f_{1}(IJ)\stackrel{\eqref{422342552343236}}{=}\log\left(\frac{\kappa^+_{I,J}}{\kappa^-_{I,J}}\right)\to 0 \quad\text{uniformly in $I$}
\end{align*} 
by \eqref{M:kappa_to_one}, we obtain even
\begin{equation}\label{alpha_m_to_one}
\alpha_m\to 0 \quad\text{as}\quad m\to\infty.
\end{equation}

Observe also that \eqref{uppergammamalt} implies \eqref{uppergammam}.
For $ |I|=1 $ this is clear. 
If $|I|>1 $, then
\begin{equation*}
  \begin{aligned}
    S_{|I|-1}f_{1}(IJ)&\stackrel{{ \eqref{24590760w6459}}}{\geq} S_{k}f_{1}(IJ)=\sum_{n=0}^{k-1} f_{1}(\sigma^n (IJ))\\
    &\stackrel{\eqref{monotonefi}{\color{white}
        a}}{\geq}\sum_{n=0}^{k-1} f_{1}(\sigma^n
    (IJ|_{|J|+k}))=S_{k}f_{1}(IJ|_{|J|+k})
  \end{aligned}
\end{equation*}
for $k=0,\dots, |I|-1$.
Next, if we define
\begin{equation*}
  \De_{\eps}^m\coloneqq  \left\{IJ: |J|=m,\; \text{ \eqref{lowergammam} and \eqref{uppergammam} hold for }IJ \right\},
\end{equation*}
\begin{equation*}
  \De_{\eps}^{'m}\coloneqq \left\{IJ: |J|=m,\; \text{ \eqref{lowergammam} and \eqref{uppergammamalt} hold for }IJ \right\},
\end{equation*}
\begin{equation*}
  \De_{\eps}(J)\coloneqq \left\{I: \; \text{ \eqref{lowergammam} and \eqref{uppergammam} hold for }IJ\right\},
\end{equation*}
\begin{equation*}
  \De'_{\eps}(J)\coloneqq \left\{I: \text{ \eqref{lowergammam} and \eqref{uppergammamalt} hold for }IJ\right\},
\end{equation*}
and
\begin{equation*}
  \De''_{\eps}(J)\coloneqq \De_{\eps}(J)\smallsetminus \De'_{\eps}(J),
\end{equation*}
we have
\begin{equation*}
\De''_{\eps}(J)\cup \De'_{\eps}(J)=\De_{\eps}(J).
\end{equation*}

Note that for all $n\ge m> m_0$ and
$\omega\in\mathcal{T}$
\begin{align*}
S_{n-m}f_1(\omega|_n)=\sum_{k=0}^{n-m-1}f_1(\sigma^k(\omega|_n))\ge(n-m)\inf_{I\in\mathcal{K}_0} \tilde f_1(I),
\end{align*}
where $\mathcal{K}_0\coloneqq\{I\in\T\cup\T^*:|I|>m_0\}$.
One has by Lemma \ref{compactnesslemma} that $\mathcal{K}_0$ is compact.
Using \eqref{strictf} and continuity of $\tilde f_1$, this yields
\begin{align*}
\inf_{I\in\mathcal{K}_0}\tilde  f_1(I)=\min_{I\in\mathcal{K}_0}\tilde f_1(I)>0.
\end{align*}
Thus $S_{n-m}f_1(\omega|_n)\to\infty$ as $n\to \infty$. 
Since by \eqref{strictf} and \eqref{monotonefi}
\begin{equation*}
S_{n-m}f_1(\omega|_n)>S_{n-m-1}f_1(\omega|_{n-1})
\end{equation*}
holds whenever $m>m_0$,
we can always obtain the smallest $n=n(\omega,m)\ge m$ such that
\begin{align*}
  S_{n-m}f_1(\omega|_n)>-\log(\eps)-\gamma_{ m }
\end{align*}
and
\begin{align*}
  S_{k}f_{1}((\omega|_n)|_{ m +k})\leq-\log(\eps)-\gamma_{ m },\; k=0,\dots, n-m-1.
\end{align*}
This yields $\omega|_n\in \De_{\eps}^m$ and consequently
\begin{align}\label{K_can_be_covered_nicely}
  K\subseteq \bigcup_{L\in \De_{\eps}^m}K_L  
\end{align}
for all $m>m_0$.
Fix $ m>m_0 $, $|J|=m$ and $I\in  \De_{\eps}(J)$. Recall the definition of $ \tilde\eps_{|J|} $, given in \eqref{eps_tilde}.
First we want to show, that if we assume $\gamma_{|J|}$ big enough, we obtain
\begin{equation}\label{eq:lessserThanEps_m}
0<\eps e^{S_{|I|}f_1(IJ)}\le\eps e^{S_{|I|}f_2(IJ)}<\tilde\eps_{|J|}\rho^{|J|}.
\end{equation}
To see this, write
\begin{equation*}
\begin{aligned}
\eps e^{S_{|I|}f_1(IJ)}&\stackrel{\eqref{monotonefi}}{\le}\eps e^{S_{|I|}f_2(IJ)}\\
		&=\eps e^{S_{|I|-1}f_2(IJ)}e^{f_2(\sigma^{|I|-1}(IJ))}\\
		&\le\eps e^{S_{|I|-1}f_2(IJ)}\frac{r_{J}}{r_{I_{|I|}J}\kappa^{-}_{I_{|I|},J}}\\
		&\le \eps e^{S_{|I|-1}f_2(IJ)} \frac{R^{|J|}}{\rho^{|J|+1}}\kappa.
\end{aligned}
\end{equation*}
Define $\tau_{|J|}\coloneqq \frac{R^{|J|}}{\rho^{|J|+1}}\kappa $. Then
\begin{align*}
\eps\tau_{|J|} e^{S_{|I|-1}f_2(IJ)}
&\stackrel{\eqref{monotonefi}}{\le}\eps\tau_{|J|} e^{S_{|I|-1}f_2(IJ|_{|I|+|J|-1})}\\
&\le \eps\tau_{|J|} e^{S_{|I|-1}f_1(IJ|_{|I|+|J|-1})+\alpha_{|J|}}\\
&\le \tau_{|J|}e^{-\gamma_{|J|}}e^{\alpha_{|J|}}.
\end{align*}
Here in the last inequality we used that $I\in  \De_{\eps}(J)$.
Thus \eqref{eq:lessserThanEps_m} holds, if
\begin{equation}\label{eq:assumptionOnGamma_m}
\frac{\tau_{|J|}e^{\alpha_{|J|}}}{\tilde\eps_{|J|}\rho^{|J|}}\le e^{\gamma_{ |J| }}.
\end{equation}
Note, that if \eqref{eq:lessserThanEps_m} holds, we have in particular
\begin{align}\label{eps_r_IJ_upper_bound}
\frac{\eps}{r_{IJ}}\stackrel{\eqref{eq:lessserThanEps_m}}{<}\frac{\tilde\eps_{|J|}\rho^{|J|}e^{-S_{|I|}f_2(IJ)}}{r_{IJ}}\stackrel{\eqref{SIf}}{=}\frac{\tilde\eps_{|J|}
\rho^{|J| }}{r_J}\kappa^{-}_{I,J}\le \kappa \tilde\eps_{|J|}.
\end{align}
As there are only finitely many $ J $ with $ |J|=m $, this implies that
\begin{align}\label{D_eps_m_is_finite}
\#\De_{\eps}^m<\infty
\end{align}
for all $ m > m_0 $.

Now assume \eqref{eq:assumptionOnGamma_m} holds for all $J$ such that $|J|>m_0$. We obtain then the following estimate:
\begin{equation}\label{upperExpression}
  \begin{aligned}
    N(\eps, K)&\stackrel{\eqref{K_can_be_covered_nicely}}{\leq}
    \sum_{\stackrel{I\in \De_{\eps}(J)}{|J|=m}} N(\eps , K_{IJ})
    \stackrel{\eqref{packingMIrev}}{\leq} \sum_{\stackrel{I\in \De_{\eps}(J)}{|J|=m}} N(\eps e^{S_{|I|}f_1(IJ)}, K_J)\\
    &\stackrel{{\color{white}\eqref{packingMIrev}}}{=}\sum_{\stackrel{I\in
        \De'_{\eps}(J)}{|J|=m}} N(\eps e^{S_{|I|}f_1(IJ)}, K_J)
    +\sum_{\stackrel{I\in
        \De''_{\eps}(J)}{|J|=m}} N(\eps e^{S_{|I|}f_1(IJ)}, K_J)\\
    &\stackrel{{\color{white}\eqref{packingMIrev}}}{\coloneqq}U_{\eps}^{m}+Q_{\eps}^{m}.
  \end{aligned}
  \end{equation}
  Similarly, one has
\begin{equation}\label{lowerExpression}
  \begin{aligned}
    N(\eps, K)&= \sum_{\stackrel{I\in \De'_{\eps}(J)}{|J|=m}} N(\eps , K_{IJ})-\Biggl(\sum_{\stackrel{I\in \De'_{\eps}(J)}{|J|=m}} N(\eps , K_{IJ})-N(\eps, K) \Biggr)\\
    &\geq \sum_{\stackrel{I\in \De'_{\eps}(J)}{|J|=m}} N(\eps
    e^{S_{|I|}f_1(IJ)+\alpha_{m}}, K_J)-R_{\eps}^{m}\\
    &=:L_{\eps}^{m}-R_{\eps}^{m}.
  \end{aligned}
\end{equation}
Here we define
\begin{equation}
  R_{\eps}^m\coloneqq \sum_{L\in \De_{\eps}^{'m}} N(\eps , K_{L})-N(\eps, K).
\end{equation}
Note that by \eqref{D_eps_m_is_finite}, above sums are always finite.

Now the following two lemmas imply Theorem~\ref{MainTheoren}:
\begin{lem}\label{L:estimateMainTerms}
  There are $0<U_{m}<\infty$ and $0<L_{m}<\infty$, $m>m_0$, such that
  \begin{align*}
    \eps^{s}U_{\eps}^m\to U_m,\quad \eps^{s}L_{\eps}^m\to L_m
  \end{align*}
  as $\eps\to 0+$ and
  \begin{equation}\label{goes_to_zero}
    \left|U_m-L_m\right|\to0
  \end{equation}
  as $m\to\infty$.
\end{lem}

\begin{lem}\label{L:estimateNegligible}
  There are $Q_{m},R_m$, $m>m_0$, such that
  \begin{equation}\label{limsup_estimate}
    \limsup_{\eps\to 0+}\eps^{s}Q_{\eps}^m\leq Q_m\quad\text{and}\quad \limsup_{\eps\to 0+}\eps^{s}R_{\eps}^m\leq R_m,
  \end{equation}
  and
  \begin{equation}\label{error_goes_to_zero}
    Q_m\to 0\quad\text{and}\quad R_m\to 0
  \end{equation}
  as $m\to\infty$.
\end{lem}

\begin{proof}[Proof of Theorem~\ref{MainTheoren}]
  By \eqref{upperExpression} and \eqref{lowerExpression} we have for all $m>m_0$
  \begin{equation*}
    \begin{aligned}
      L_m-R_m
      &\stackrel{\eqref{limsup_estimate}}{\leq} \liminf_{\eps\to0+}\eps^s L_{\eps}^m-\limsup_{\eps\to0+}\eps^s R_{\eps}^m\\
      &\stackrel{{\color{white}\eqref{error_goes_to_zero}}}{\leq} \liminf_{\eps\to0+} \eps^s N(\eps,K)\\
      &\stackrel{{\color{white}\eqref{error_goes_to_zero}}}{\leq} \limsup_{\eps\to0+} \eps^s N(\eps,K)\\
      &\stackrel{{\color{white}\eqref{error_goes_to_zero}}}{\leq} \limsup_{\eps\to0+} \eps^s U_{\eps}^m+\limsup_{\eps\to0+} \eps^s Q_{\eps}^m\\
      &\stackrel{\eqref{limsup_estimate}}{\leq} U_m+Q_m.
    \end{aligned}
  \end{equation*}
  Now there exists a subsequence $\{m_k\}_{k\in\N}$ such that $\lim_{k\to\infty}L_{m_k}= \theta \in[-\infty,\infty]$.
  Hence by \eqref{goes_to_zero} we obtain $\lim_{k\to\infty}U_{m_k}= \theta $ and
  \begin{align*}
    \theta &\stackrel{\eqref{error_goes_to_zero}}{=} \lim_{k\to\infty}L_{m_k}-R_{m_k}\\
    &\stackrel{{\color{white}\eqref{error_goes_to_zero}}}{\leq}\liminf_{\eps\to0+} \eps^s N(\eps,K)\\
    &\stackrel{{\color{white}\eqref{error_goes_to_zero}}}{\leq} \limsup_{\eps\to0+} \eps^s N(\eps,K)\\
    &\stackrel{{\color{white}\eqref{error_goes_to_zero}}}{\leq} \lim_{k\to\infty}U_{m_k}+Q_{m_k}\stackrel{\eqref{error_goes_to_zero}}{=}\theta .
  \end{align*}
  Using \eqref{sRegularInequalinies} this implies $\lim_{\eps\to0+} \eps^s N(\eps,K)=\theta \in(0,\infty)$,
  which completes the proof.
\end{proof}

\begin{lem}\label{L:estimateOverlaps}
  There are $\Gamma_{m}$ and $\eps_{m}$ for each $m>m_0$, such that
  $\Gamma_m\to0$ as $m\to\infty$, and
  \begin{equation}
    R_{\eps}^m\leq \eps^{-s}\Gamma_m
  \end{equation}
  for every $0<\eps<\eps_{m}$.
\end{lem}

\begin{proof}
  Put for all $m> m_0$
  \begin{equation*}
    \tilde\Gamma_{m}:=\inf_{L\in\mathcal{T}_{m+1}}e^{-\alpha_{m+1}}\diam(K_{L})
  \end{equation*}
  and take $\gamma_m$ such that \eqref{eq:assumptionOnGamma_m} and
  \begin{equation}\label{how_to_choose_gamma_m}
    e^{\gamma_m}> \frac{2AD}{\Upsilon_{i}C\tilde\Gamma_{m}}
  \end{equation}
  hold for every $i=1,\dots, m+1$ (recall that $A>0$ is the constant guaranteed by \ref{C:separation}, and $ C,D $ are given by \ref{jedna} and \ref{dva}).
  The constants $\Upsilon_{i}$ are
  defined in \eqref{diam_estimate}.  Now pick $\eps_m$ in such a way that
  \begin{equation}\label{epsilon_m}
    f_1(J)\le-\log(\eps_m)-\gamma_m
  \end{equation}
  for all $J\in \T_m$.  This in particular implies that for all
  $\eps<\eps_m$
  \begin{equation}\label{IlongEnough}
    |I|\geq m+1\quad\text{whenever}\quad I\in\De_{\eps}^{'m}
    \end{equation}
  For the rest of the proof, fix $m>m_0$ and $\eps_m>\eps>0$.

  \textit{Step 1.} Fix $IJ\in\De_{\eps}^{'m}$, where $|J|=m$. First
  note that
  \begin{equation}\label{diameter_estimate}
    \diam K_{IJ}\geq \tilde\Gamma_{m} e^{\gamma_{m}}\eps.
  \end{equation}
  Indeed, by \eqref{uppergammamalt} we can write
  \begin{equation*}
    \begin{aligned}
      \eps e^{\gamma_m}&\stackrel{{\eqref{IlongEnough}}}{\leq} e^{- S_{|I|-1}f_{1}(IJ)}=e^{\alpha_{m+1}}e^{- S_{|I|-1}f_{1}(IJ)-\alpha_{m+1}}\leq  e^{\alpha_{m+1}}e^{- S_{|I|-1}f_{2}(IJ)}\\
      &\stackrel{{\color{white}\eqref{diameterMIrev}}}{=} \frac{e^{\alpha_{m+1}}}{\diam(K_{\sigma^{|I|-1}IJ})}e^{- S_{|I|-1}f_{2}(IJ)}\diam(K_{\sigma^{|I|-1}IJ})\\
      &\stackrel{\eqref{diameterMIrev}}{\leq}
      \frac{e^{\alpha_{m+1}}}{\diam(K_{\sigma^{|I|-1}IJ})}\diam(K_{IJ})\\
      &\stackrel{{\color{white}\eqref{diameterMIrev}}}{\le}\sup_{L\in\mathcal{T}_{m+1}}\frac{e^{\alpha_{m+1}}}{\diam(K_{L})}\diam(K_{IJ})=
      \frac{1}{\tilde\Gamma_{m}}\diam(K_{IJ}).
    \end{aligned}
  \end{equation*}
  Hence for all $L\in\mathcal{T}^*$ one has
  \begin{equation*}
    \frac{C}{2}r_{IJL}\stackrel{\ref{dva}}{\geq} \frac{C\diam K_{IJL}}{2D}\stackrel{\eqref{diam_estimate}}{\geq} \Upsilon_{|L|}\frac{C\diam K_{IJ}}{2D}\stackrel{\eqref{diameter_estimate}}{\geq} \Upsilon_{|L|}\frac{C\tilde\Gamma_{|J|}  e^{\gamma_{|J|}}\eps}{2D}.
  \end{equation*}      
  Due to \eqref{copy_in_ball} we know that
  \begin{align*}
    K_{IJL\phi(IJL)}\subseteq B\left(x_{IJL},\frac{Cr_{IJL}}{2}\right).
  \end{align*}
  Assuming that $\tilde L\in\mathcal{T}^*$ is incomparable with
  $IJL$, condition \ref{jedna} gives
  \[
    d(x_{IJL},x_{\tilde L\tilde J})\ge Cr_{IJL}
  \]
  for all $\tilde J\in\mathcal{T}^*$. This yields
  \begin{equation}\label{farEnough}
    d(K_{IJL\phi(IJL)},K_{\tilde L})\geq \frac{C}{2}r_{IJL}\stackrel{\eqref{how_to_choose_gamma_m}}{>}A\eps.
  \end{equation}

  \textit{Step 2.} 
  For $I\in \De_{\eps}^{'m}$ define
  \begin{equation*}
  \widetilde{\T}_{k}^*(I)\coloneqq\{J\in\T^*:\;|J|=k,\; IM\phi(IM)\nprec IJ\quad\text{for any}\quad M\in\T^*\}.
  \end{equation*}
  Note that
  $\widetilde{\mathcal{T}}^*(I)=\bigcup_{k\ge
  	0}\widetilde{\T}_{k}^*(I)$, where $\widetilde{\mathcal{T}}^*(I)$
  was defined in Lemma~\ref{prunnedTreeLemma}.
  Observe that all elements in $\De_{\eps}^{'m}$ are
  pairwise incomparable.  Using this, we can write
  \begin{equation}
    \begin{aligned}
      R_{\eps}^m\stackrel{{\color{white}\eqref{diameterMIrev}}}{=}&\sum_{I\in \De_{\eps}^{'m}} N(\eps , K_{I})-N(\eps,
      K)\\
      \stackrel{{\color{white}\eqref{diameterMIrev}}}{\leq}&\sum_{I\in \De_{\eps}^{'m}} N(\eps , K_{I})-N\left(\eps, \bigcup_{\substack{I\in \De_{\eps}^{'m} \\ |L|\leq m  \\L\in \widetilde{\mathcal{T}}^*(I)}} K_{IL\phi(IL)}\right)\\
      \stackrel{\eqref{farEnough}}{=}&\sum_{I\in \De_{\eps}^{'m}} N(\eps , K_{I})-\sum_{\substack{I\in \De_{\eps}^{'m} \\ |L|\leq m \\ L\in \widetilde{\mathcal{T}}^*(I)}} N\left(\eps,  K_{IL\phi(IL)}\right)\\
      \stackrel{{\color{white}\eqref{diameterMIrev}}}{=}&\sum_{I\in \De_{\eps}^{'m}} \left(N(\eps ,
        K_{I})-\sum_{\substack{|L|\leq m \\ L\in \widetilde{\mathcal{T}}^*(I)}} N\left(\eps,
          K_{IL\phi(IL)}\right)\right)=:\sum_{I\in \De_{\eps}^{'m}}
      R_{\eps}^{m}(I).
    \end{aligned}
    \end{equation}
    Here in the second equality, we applied \eqref{farEnough} to \ref{C:separation}.
  Fix $I\in \De_{\eps}^{'m}$.  We claim that
  \begin{equation}\label{inductionInequality}
    \begin{aligned}
      \sum_{J\in \widetilde{\T}_{k}^*(I)} N(\eps , K_{IJ})&-\sum_{\substack{k\leq |L|\leq m \\ L\in \widetilde{\mathcal{T}}^*(I)}} N\left(\eps,  K_{IL \phi(IL )}\right)\\
      &\leq \sum_{J\in \widetilde{\T}_{k+1}^*(I)} N(\eps ,
      K_{IJ})-\sum_{\substack{k+1\leq |L|\leq m \\ L\in \widetilde{\mathcal{T}}^*(I)}} N\left(\eps,
        K_{IL \phi(IL )}\right),
    \end{aligned}
  \end{equation}
  whenever $0\leq k\leq m$.  Indeed, we can write
  \begin{equation*}
    \begin{aligned}
      \sum_{J\in \widetilde{\T}_{k}^*(I)} N(\eps , K_{IJ})&-\sum_{\substack{k\leq |L|\leq m \\ L\in \widetilde{\mathcal{T}}^*(I)}} N\left(\eps,  K_{IL \phi(IL )}\right)\\
      &\leq \sum_{J\in \widetilde{\T}_{k}^*(I)}\sum_{i} N(\eps , K_{IJi})-\sum_{\substack{k\leq |L|\leq m \\ L\in \widetilde{\mathcal{T}}^*(I)}} N\left(\eps,  K_{IL \phi(IL )}\right)\\
      &\leq \sum_{J\in \widetilde{\T}_{k}^*(I)}\sum_{i\not=\phi(IJ)} N(\eps , K_{IJi})-\sum_{\substack{k+1\leq |L|\leq m \\ L\in \widetilde{\mathcal{T}}^*(I)}} N\left(\eps,  K_{IL \phi(IL )}\right)\\
      &= \sum_{J\in \widetilde{\T}_{k+1}^*(I)} N(\eps ,
      K_{IJ})-\sum_{\substack{k+1\leq |L|\leq m \\ L\in \widetilde{\mathcal{T}}^*(I)}} N\left(\eps,
        K_{IL \phi(IL )}\right).
    \end{aligned}
  \end{equation*}
  Using \eqref{inductionInequality} $m+1$ times we obtain
  \begin{equation}\label{estimateRI}
    R_{\eps}^{m}(I)\leq \sum_{J\in \widetilde{\T}_{m+1}^*(I)} N(\eps , K_{IJ})\stackrel{\ref{C:comparability}}{\leq}B \sum_{J\in \widetilde{\T}_{m+1}^*(I)}  \eS\left(\frac{\eps}{B}, K_{IJ}\right)
  \end{equation}
  for all $I\in \De_{\eps}^{'m}$.

  \textit{Step 3.} 
  {\color{black} At the end of the previous step we essentially reduced the general case to the situation, where $N=\eS$. 
  This does not have any deep meaning, but serves one important technical purpose. Since the constant $G$ (from \ref{C:lipschitzImage}) equals to $0$ in the case of function $\eS$, we can pick $\xi>0$ arbitrarily, and so equation \eqref{packingMI} holds for any $\eps>0$. This makes some estimates later in the proof slightly easier, since we do not have to worry about $\eps$ being small enough. In what follows we also assume that the constants (mainly $\kappa_{I,J}^\pm$ and $\kappa$) are the ones corresponding to the counting function $\eS$}.

  As a next step we claim that there is a constant $Q$
  independent of $m$ such that
  \begin{equation}\label{comparationEstimate}
     B \eS\left(\frac{\eps}{B}, K_{IJ}\right)\leq Q\frac{r_{IJ}^s}{r_I^s} \eS(\eps, K_{I}).
  \end{equation}

  For this one can write 
  \begin{equation*}
    \begin{aligned}
      \eS\left(\frac{\eps}{B}, K_{IJ}\right)
      &\stackrel{\eqref{packingMI}}{\leq} \eS\left(\frac{\eps }{\kappa^+_{IJ,\varnothing} r_{IJ}B}, K\right)\\
      &\stackrel{L\,\ref{L:differentEpsilons}}{\leq}R\left(\frac{\kappa\tilde\eps_m }{\rho^{m+1}}\kappa\right) \left(B\frac{\kappa^+_{IJ,\varnothing}r_{IJ}}{\kappa^-_{I,\varnothing} r_I}\right)^s \eS\left(\frac{\eps}{\kappa^-_{I,\varnothing} r_{I}}, K\right)\\
      &\stackrel{\eqref{packingMI}}{\le}R\left(\frac{\xi }{\rho}\right)B^s\kappa^{2s}\frac{r_{IJ}^s}{r_I^s}\eS(\eps,
      K_{I}).
    \end{aligned}
    \end{equation*}
    And so we can put $Q\coloneqq BR\left(\frac{\xi }{\rho}\right)B^s\kappa^{2s}$.

  \textit{Step 4.} Combining \eqref{estimateRI} and
  \eqref{comparationEstimate} and Lemma~\ref{prunnedTreeLemma} we can
  write for all $I\in \De_{\eps}^{'m}$
  \begin{equation*}
    \begin{aligned}
      R_{\eps}^{m}(I)&\stackrel{\eqref{estimateRI}}{\leq} \sum_{J\in
        \widetilde{\T}_{m+1}^*(I)}\eS\left(\frac{\eps}{B}, K_{IJ}\right)
      \stackrel{\eqref{comparationEstimate}}{\leq} Q   \sum_{J\in \widetilde{\T}_{m+1}^*(I)} \frac{r_{IJ}^s}{r_I^s} \eS(\eps, K_{I})\\
      &\stackrel{L\,\ref{prunnedTreeLemma}}{\leq} Q E^2 (1-\rho^s)^{m+1}  \eS(\eps, K_{I})\\
      &\stackrel{\eqref{packingMI}}{\leq} Q E^2  (1-\rho^s)^{m+1}   \eS\left(\frac{\eps }{\kappa^+_{I,\varnothing} r_I},K\right)\\
      &\stackrel{L\,\ref{L:differentEpsilons}}{\leq} QE^2  (1-\rho^s)^{m+1} R(\xi \kappa)  \left(\kappa r_{I}\right)^s  \eS(\eps, K)\\
      &\stackrel{\eqref{sRegularInequalinies}}{\leq} QE^2 
      (1-\rho^s)^{m+1} R(\xi\kappa)   W 
      r_{I}^s\eps^{-s}\\
      &\stackrel{{\color{white}\eqref{sRegularInequalinies}}}{\eqqcolon}
      M  (1-\rho^s)^{m+1}r_{I}^s\eps^{-s}.
    \end{aligned}
    \end{equation*}
 In the last inequality the
  constant $W$ is derived from the $\limsup$ in
  \eqref{sRegularInequalinies}.  Finally, we have
  \begin{equation*}
    \sum_{I\in \De_{\eps}^{'m}} R_{\eps}^{m}(I)\leq M   (1-\rho^s)^{m+1}\eps^{-s} \sum_{I\in \De_{\eps}^{'m}} r_{I}^s\stackrel{\ref{triAlt}}{\leq} M E  (1-\rho^s)^{m+1}\eps^{-s}=:\eps^{-s}\Gamma_m,
    \end{equation*}
  which is what we wanted.
\end{proof}
\begin{lem}\label{useful_estimate_stuff}
 There exists a constant $\Lambda$, independent of $|J|$, such that
 \begin{align*}
 N\left(e^{-\gamma_{|J|}}, K_J\right)\leq\Lambda e^{s\gamma_{|J|}}r_J^s.
 \end{align*}

\end{lem}
\begin{proof}
{\color{black}Similarly to the previous lemma we can reduce the general case to $N=\eS$.}
One can write
          \begin{equation*}
            \begin{aligned}
              N\left(e^{-\gamma_{|J|}}, K_J\right)
              \stackrel{\ref{C:comparability}}{\leq} B \eS\left(\frac{e^{-\gamma_{|J|}}}{B}, K_J\right)
              \stackrel{\eqref{packingMI}}{\leq} B \eS\left(\frac{e^{-\gamma_{|J|}} }{Br_{J}\kappa^+_{J,\varnothing}}, K\right)
              &\stackrel{\eqref{M:kappa}}{\leq}
              B\eS\left(\frac{e^{-\gamma_{|J|}} }{r_{J}\kappa}, K\right).
            \end{aligned}
            \end{equation*}
          Note that by \eqref{eq:lessserThanEps_m} and \eqref{eps_tilde} one has
          $\frac{e^{-\gamma_{|J|}} }{r_{J}\kappa}\leq \frac{\tilde{\eps}_{|J|} }{\rho^{|J|}\kappa} \le \xi$.
          Applying
          Lemma \ref{L:differentEpsilons}
          \begin{align*}
            \eS\left(\frac{e^{-\gamma_{|J|}} }{r_{J}\kappa},K\right)\stackrel{L\,\ref{L:differentEpsilons}}{\leq}
            R\left(\xi \right)\frac{\xi^s r_{J}^s\kappa^s}{e^{-s\gamma_{|J|}}}\eS\left(\xi,
            K\right).
          \end{align*}
          This implies the statement of the lemma.
\end{proof}

From now on, the proof is identical to the one in \cite{Lalley1989}.
For the convenience of the reader, we rewrite the main points here.

	\begin{lem}
        There is $ \Xi >0$ independent of
          $J\in\bigcup_{k>m_0}^\infty\mathcal{T}_k$ and $\eps_{m}$,
          $m\in\en$, such that
          \begin{equation}\label{card_of_dddash}
            \# \De''_{\eps}(J)\leq \eps^{-s} \Xi e^{-s\gamma_{|J|}}\alpha_{|J|}
          \end{equation} 
          for every $0<\eps<\eps_{|J|}$.
	\end{lem}
	
	\begin{proof}
          This is exactly like the proof of \cite[Lemma
          13.6]{Lalley1989}. 
	\end{proof}

	\begin{proof}[Proof of Lemma~\ref{L:estimateNegligible}]
          For $I\in \De''_{\eps}(J)$, $|J|=m>m_0$, we have
          \begin{align*}
            N\left(\eps e^{S_{|I|}f_1(IJ)}, K_J\right)\stackrel{\ref{C:monotonicity}}{\le} N\left(e^{-\gamma_{|J|}}, K_J\right).
          \end{align*}
         Thus by Lemma \ref{useful_estimate_stuff} we obtain for all $m>m_0$
          \begin{equation*}
            \begin{aligned}
              \eps^sQ_{\eps}^m&\stackrel{{L\,\ref{useful_estimate_stuff}}}{\leq}\eps^s \Lambda e^{s\gamma_{m}}\sum_{\stackrel{I\in \De''_{\eps}(J)}{|J|=m}}r_{J}^s\leq\eps^s  \Lambda e^{s\gamma_{m}}\sum_{|J|=m}r_{J}^s\# \De''_{\eps}(J)\\
              &\stackrel{\eqref{card_of_dddash}}{\leq}
              \eps^s\eps^{-s}\Lambda e^{s\gamma_{m}} \Xi
              e^{-s\gamma_{m}}\alpha_{m}\sum_{|J|=m}r_{J}^s\stackrel{\ref{triAlt}}{\le}\Lambda \Xi
              \alpha_{m}E =:Q_{m}.
            \end{aligned}
          \end{equation*}
          This is enough since by \eqref{alpha_m_to_one} we have that $\alpha_{m}\to 0$ as $m\to\infty$.
          The statements involving $R_{\eps}^m$ follow directly from
          Lemma~\ref{L:estimateOverlaps}.
	\end{proof}

	\begin{proof}[Proof of Lemma~\ref{L:estimateMainTerms}]
		
          Pick $J\in\bigcup_{k>m_0}^\infty\mathcal{T}_k$
          and consider the functions $G_J,\tilde G_J:\er\to\er$
          defined by
          \begin{align}\label{def_of_G_J}
            G_{J}(t)\coloneqq N(e^{t-\gamma_{|J|}}, K_{J})\quad\text{and}\quad\tilde G_{J}(t)\coloneqq G_J(t+\alpha_{|J|})
          \end{align}
          Then for $\eps<\eps_{|J|}$, where $\eps_{|J|}$ is chosen according to \eqref{epsilon_m},
          \begin{align*}
            \sum_{I\in \De'_{\eps}(J)} &N\left(\eps
            e^{S_{|I|}f_1(IJ)}, K_J\right)
            \stackrel{\eqref{IlongEnough}}{=}\sum_{i}\sum_{Ii\in \De'_{\eps}(J)} N\left(\eps e^{S_{|I|+1}f_1(IiJ)}, K_J\right)\\
            &= \sum_{i}\sum_{Ii\in \De'_{\eps}(J)} G_{J}\left(S_{|I|+1}f_1(IiJ)-(-\log(\eps)-\gamma_{|J|})\right)\\
            &= \sum_{i}\sum_{I} G_{J}\left(S_{|I|+1}f_1(IiJ)-(-\log(\eps)-\gamma_{|J|})\right)\chi_{P(iJ)}\left(-\log(\eps)-\gamma_{|J|},I\right)\\
            &= \sum_{i}
               N_{G_{J}}\left(-\log(\eps)-\gamma_{|J|},iJ\right).
          \end{align*}
          Note that 
          and similarly
          \begin{equation}
            \sum_{I\in \De'_{\eps}(J)} N\left(\eps e^{S_{|I|}f_1(IJ)+\alpha_{|J|}}, K_J\right)
            =\sum_{i} N_{\tilde G_{J}}\left(-\log(\eps)-\gamma_{|J|},iJ\right).
          \end{equation}
          Here we denoted
          \begin{align*}
            &P(L)\coloneqq \left\{(t,I): S_{|I|+1}f_1(IL)>t\geq S_{j}f_1(IL),\; j\leq |I|\right\}
          \end{align*}
          and
          \begin{align*}
            N_{G}(a,L)\coloneqq \sum_{I} G\left(S_{|I|+1}f_1(IL)-a\right)\chi_{P(L)}(a,I),
          \end{align*}
          for $G:\R\to\R$,
          like in \eqref{P(L)} and \eqref{N_G(a,L)}.
	Then, using \eqref{strictf}, by Proposition \ref{renewal_theorem} one has for all $m>m_0$
	\begin{align*}
        \eps^sU^m_\eps&= \eps^s\sum_{i,|J|=m}N_{G_{J}}\left(-\log(\eps)-\gamma_{|J|},iJ\right)\\
        &\to \sum_{i,|J|=m}e^{-s\gamma_{|J|}}\int\limits_{0}^{\infty}G_{J}(t)F(iJ,dt)\eqqcolon U_m
	\end{align*}
	and
	\begin{align*}
        \eps^sL^m_\eps&=\sum_{i,|J|=m}N_{\tilde G_{J}}\left(-\log(\eps)-\gamma_{|J|},iJ\right)\\
        &\to \sum_{i,|J|=m}e^{-s\gamma_{|J|}}\int\limits_{0}^{\infty}\tilde G_{J}(t)F(iJ,dt)\eqqcolon L_m
	\end{align*}
        as $\eps\to0$.
        Note that all the integrals above (and so also $U_m$ and $L_m$) are finite, since we integrate a bounded function with respect to a measure with a compact support.
        Next observe, that by \ref{C:monotonicity} one has
        \begin{align*}
          0\le\frac{N(e^{\|f_1\|_\infty+1-\gamma_{|J|}}, K_{J})}{N(e^{-\gamma_{|J|}}, K_{J})}\le1,
        \end{align*}
        thus
        \begin{align*}
          \left|\frac{G_J(\|f_1\|_\infty+1)}{G_J(0)}-\frac{G_J(0)}{G_J(0)}\right|\le 1.
        \end{align*}
       By \cite[Corollary 3.3]{Lalley1989} we obtain that for any $\delta>0$ there exists an $\alpha_*(\delta)>0$ such that for all $0<\alpha<\alpha_*(\delta)$ one has
        \begin{align*}
          \left|\int_0^\infty\frac{G_J(t+\alpha)}{G_J(0)}F(L,dt)-\int_0^\infty\frac{G_J(t)}{G_J(0)}F(L,dt)\right|<\delta
        \end{align*}
        for all $L\in\bigcup_{k>m_0}^\infty\mathcal{T}_k$.
        This can be rephrased as
        \begin{align*}
          \left|\int_0^\infty\frac{\tilde G_J(t)}{G_J(0)}F(L,dt)-\int_0^\infty\frac{G_J(t)}{G_J(0)}F(L,dt)\right|<\delta_{|J|},
        \end{align*}
        for all $L\in\bigcup_{k>m_0}^\infty\mathcal{T}_k$, where $\delta_{|J|}\to0$ as $|J|\to\infty$. 
        This gives
        \begin{align*}
          \left|U_m-L_m\right|&\le e^{-s\gamma_{m}}\delta_m\sum_{i, |J|=m}G_J(0)\\
          &= e^{-s\gamma_{m}}\delta_m\sum_{i, |J|=m}N(e^{-\gamma_{|J|}}, K_{J})\stackrel{L\,\ref{useful_estimate_stuff}}{\le}\delta_m
          \Lambda \sum_{i, |J|=m}r_J^s\stackrel{\ref{triAlt}}{\le} \delta_m\Lambda N E .
        \end{align*}
        Here $N$ is the cardinality of the alphabet of $\T$.
        Hence $ \left|U_m-L_m\right|\to0$ as $m\to\infty$.
	\end{proof}
	
\subsection{Minkowski measurability}\label{S:Minkowski}
The framework in Section~\ref{SS:setup} unfortunately does not cover arguably the most interesting case, which is the question of Minkowski measurability.
In this section we restrict our interest only to the sets $K\subseteq\er^d$.
Recall that the counting function we consider in order to study Minkowski measurability is of the form
\begin{equation*}
\eM(\eps,K)=\frac{|K_\eps|}{\eps^d}.
\end{equation*}
As mentioned in section~\ref{SS:minkowskiMeasurability} we can no longer use \ref{C:lipschitzImage} which needs to be replaced by \ref{C:biLipschitzImage}.
By \eqref{neighborhood_inclusion} we know that
$\varphi_I $ is $(\kappa^{-}_{I,J}\frac{r_{IJ}}{r_J})-(\kappa^{+}_{I,J}\frac{r_{IJ}}{r_J} )$-Lipschitz
on $ (K_J)_{\frac{G\xi r_J}{\kappa  }}$. For $ \eps'\le \frac{\xi r_J}{\kappa  }$, property  \ref{C:biLipschitzImage} yields
\begin{align}\label{intermediate_step_c4'}
\eM(\eps',K_J)\ge \left(\frac{\kappa^{-}_{I,J}}{\kappa^{+}_{I,J}}\right)^d\eM\left(\eps'\kappa^{-}_{I,J}\frac{r_{IJ}}{r_J},K_{IJ}\right).
\end{align}
Now define $ \eps\coloneqq  \eps'\kappa^{-}_{I,J}\frac{r_{IJ}}{r_J}$.
By \eqref{eps_r_IJ_upper_bound} we have then
  \begin{align*}
  \eps' = \eps e^{S_{|I|}f_1(IJ)} \le \tilde \eps_{|J|}=\frac{\xi\rho^{|J|}}{\kappa^2 }\le \frac{\xi r_{|J|}}{\kappa  }.
  \end{align*}
That means equation \eqref{intermediate_step_c4'} can be written as
\begin{align}\label{intermediate_step_c4'_part_2}
\eM\left(\eps,K_{IJ}\right)\le \left(\frac{\kappa^{+}_{I,J}}{\kappa^{-}_{I,J}}\right)^d\eM\left(\eps\frac{r_J}{r_{IJ}\kappa^{-}_{I,J}},K_J\right).
\end{align}
This leads to replacing \eqref{upperExpression} and \eqref{lowerExpression} by
\begin{equation}\label{eq:upperExpressionMM}
\begin{aligned}
\eM(\eps,K)&\leq \sum_{\stackrel{I\in \De_{\eps}(J)}{|J|=m}} \eM(\eps,K_{IJ})
\leq \sum_{\stackrel{I\in \De_{\eps}(J)}{|J|=m}} \left(\frac{\kappa^{+}_{I,J}}{\kappa^{-}_{I,J}}\right)^d \eM\left(\eps\frac{r_J}{r_{IJ}\kappa^{-}_{I,J}},K_J\right)\\
&\leq e^{d\alpha_m}  \sum_{\stackrel{I\in \De_{\eps}(J)}{|J|=m}} \eM\left(\eps e^{S_{|I|}f_2(IJ)},K_J\right)\\
&\leq e^{d\alpha_m}  \sum_{\stackrel{I\in \De_{\eps}(J)}{|J|=m}} \eM\left(\eps e^{S_{|I|}f_1(IJ)},K_J\right)\\
&= e^{d\alpha_m}  \sum_{\stackrel{I\in \De'_{\eps}(J)}{|J|=m}} \eM\left(\eps e^{S_{|I|}f_1(IJ)},K_J\right)
+ e^{d\alpha_m}  \sum_{\stackrel{I\in \De''_{\eps}(J)}{|J|=m}} \eM\left(\eps e^{S_{|I|}f_1(IJ)},K_J\right).
\end{aligned}
\end{equation}

and similarly obtained
\begin{equation}\label{eq:lowerExpressionMM}
\begin{aligned}
\eM(\eps,K)\geq e^{-d\alpha_m}  \sum_{\stackrel{I\in \De_{\eps}(J)}{|J|=m}} \eM\left(e^{S_{|I|}f_1(IJ)},K_J\right)-R_{\eM,\eps}^{m}.
\end{aligned}
\end{equation}
Here we denoted 
\begin{equation*}
R_{\eM,\eps}^{m}\coloneqq \sum_{L\in \De_{\eps}^{'m}} \eM(\eps , K_{L})-\eM(\eps, K).
\end{equation*}

{\color{black}
Next observe that the proofs of Lemma~\ref{L:estimateNegligible} and Lemma~\ref{L:estimateOverlaps} still go through, if we replace \ref{C:lipschitzImage} by \ref{C:biLipschitzImage} and use \ref{intermediate_step_c4'_part_2} (with the corresponding lower estimate) instead of \ref{packingMI}.

Finally, to deal with main terms is now even easier since we have the same term on both sides and therefore it suffices to use the Renewal theorem just once, 
and of course we use the fact that $\alpha_m\to 0$ as $m\to\infty$ to obtain that $e^{\pm d\alpha_m}\to 1$ and $m\to\infty$.}

\section{$\alpha$-almost similar images}\label{S:examples}

\subsection{Setup and main result}
In the previous section we found some conditions on a tree for the limit $\eps^sN(\eps,K)$ to exist.
We are now interested whether this behaviour transfers to images $\Phi(K)$ with some reasonable class of mappings $\Phi$.
We start by stating what the term reasonable mapping might mean in this context.
\begin{definition}
	Let $0<\alpha\leq 1$, $(X,d_X),(Y,d_Y)$ metric spaces and $S\ge 0$.
	A bi-Lipschitz mapping  $\Phi: X\to Y$ is said to be $\alpha$-almost similar with parameter $S$, 
	if there is a constant $A>0$, such that for every compact $\varnothing\neq K\subseteq X$ there exists  a constant $C_K$ satisfying
	\begin{equation}
	C_K \left(1+A\;\diam^\alpha K\right)^{-1}\leq \frac{d_{Y}(\Phi(x),\Phi(y))}{d_{X}(x,y)}\leq C_K \left(1+A\;\diam^\alpha K\right)
	\end{equation}
        for all $x\neq y\in K_ {S \diam K}$.
        
        Note that if $S=0$, the mapping $\Phi$ in the above definition carries only information about the intrinsic structure of $K$. Therefore in that case we can expect conclusions for the behaviour of $\eps \mapsto\eps^sN(\eps,\Phi(K))$ only for $N$ satisfying $G=0$ (cf. the assumptions of Theorem~\ref{MainTheoremImage}).
\end{definition}

\begin{rem}
	Suppose that the constant $C_K$ is chosen optimal for every $K$ and suppose that $\Phi$ is bi-Lipschitz with a constant $L>0$.
	Then 
\begin{equation}\label{eq:C_KestimaledByL}
\frac{1}{L}\leq \frac{1+A\;\diam^\alpha K}{L}\leq C_K\leq \frac{L}{1+A\;\diam^\alpha K}\leq L.
\end{equation}
Suppose additionally that $K$ and $K'$ are two compacts in $X$, then
\begin{equation*}
\frac{C_{K'}}{1+A\;\diam^\alpha {K'}}\leq \frac{C_K}{1+A\;\diam^\alpha K}\leq C_K (1+A\;\diam^\alpha K)\leq C_{K'} (1+A\;\diam^\alpha {K'})
\end{equation*}
provided $K\subseteq {K'}$.
This implies
\begin{equation}\label{eq:boundsC_K}
\frac{1+A\;\diam^\alpha K}{1+A\;\diam^\alpha {K'}}\leq \frac{C_{K'}}{C_K}\leq \frac{1+A\;\diam^\alpha {K'}}{1+A\;\diam^\alpha K}
\end{equation}
 for $K\subseteq {K'}$.
\end{rem}

Suppose that $\Phi:X\to Y$ is $\alpha$-almost similar for some $\alpha\in(0,1]$ 
and  that $\mathrm{T}=(\T,\{x_I\}_{I\in\T^*},\{r_I\}_{I\in\T^*})$ is an $s$-tree in $X$ satisfying conditions \ref{M:finiteType}-\eqref{M:holder/nonlattice} with the corresponding $s$-regular set $K$.

In this section we will carry out the setup from Section~\ref{S:measurability}.
This includes constants $C$, $D$ and $E$ from the definition of the $s$-tree (including assuming only \ref{triAlt} to hold),
constants $A$, $B$, and $G$ corresponding to the counting function $N$.
We will assume that $\tilde N$ is a counting function on $Y$ with the set of corresponding constants $\tilde A$, $\tilde B$, and $\tilde G$.
 
\begin{thm}\label{MainTheoremImage}
	Let $\mathrm{T}$ be an $s$-tree in $X$ satisfying conditions
	\ref{M:finiteType}-\eqref{M:holder/nonlattice} and suppose that $\Phi: X\to Y$ an $\alpha$-almost similar mapping with parameter $S$
	such that if $\widetilde G>0$ then $G,S>0$. 
	Then there exists
	$0<\theta<\infty$ such that
	\begin{equation*}
	\lim_{\eps\to 0+} \eps^s \tilde N(\eps,\Phi(K))=\theta.
	\end{equation*}
\end{thm}

\begin{proof}
        Without any loss of generality we may assume that $\Phi$ is onto $Y$.
	We will construct an $s$-tree $\widetilde{ \mathrm{ T}}\coloneqq(\T,\{\tilde x_I\}_{I\in\T^*},\{\tilde r_I\}_{I\in\T^*})$ in $Y$ corresponding to $\Phi(K)$.
        The statement follows then from Theorem \ref{MainTheoren}.
	First define $\tilde x_I\coloneqq \Phi(x_I)$, $I\in\T^*$.
	Next let the constants $C_K$ be optimal and let $L$ be the bi-Lipschitz constant of $\Phi$.
	By Lemma \ref{relabel_trick} we will assume that 
	\begin{equation}\label{eq:lesserThanL}
	r_I,\rho_\omega, R<\frac{1}{L^2},\quad I\in\T^*,\;\omega\in\T.
	\end{equation}
	For $I\in\T^* $ put $C_I\coloneqq C_{K_I}$ and define $\tilde r_I\coloneqq C_Ir_I$.
	Now observe that with this choice, $\widetilde{ \mathrm{ T}}$ is an $s$-tree.
	Indeed, using \eqref{eq:C_KestimaledByL} we obtain that \ref{jedna} holds with $\tilde C\coloneqq \frac{C}{L^2}$, 
	\ref{dva} holds with $\tilde D\coloneqq L^2D$,
	\ref{triAlt} holds with $\tilde E\coloneqq L^{2s}E$
	and \ref{pet} holds with $\tilde \rho\coloneqq \frac{\rho}{L^2}$, validity
        of condition \ref{ctyri} is clear.  
        The tree $\widetilde{ \mathrm{ T}}$ generates a set $\widetilde{K}$;
        note that $ \tilde K =\Phi(K)$. To prove this suppose $ x\in K $, then there is an $ \omega $ such that $ x=x_\omega=\lim_nx_{\omega|_n} $, thus $ \lim_n\tilde x_{\omega|_n} = \lim_n\Phi(x_{\omega|_n})=\Phi(x)\in \tilde K $. 
        This proves that $ \tilde K \supseteq\Phi(K)$. The opposite inclusion can be proven similarly, in fact the same argument proves also $\widetilde{K}_J=\Phi(K_J)$, $J\in\T^*$.
      
      In the next step we will
        prove that $\widetilde{ \mathrm{ T}}$ satisfies also conditions
        \ref{M:finiteType}-\eqref{M:holder/nonlattice}.
        
 Conditions \ref{M:finiteType} and \ref{M:inclusionx_I} are immediate.  Moreover, \eqref{eq:C_KestimaledByL} and \eqref{eq:lesserThanL} allow us to write
 \begin{equation*}
 \frac{\tilde r_{Ii}}{\tilde r_I}=\frac{C_{Ii}}{C_I}\frac{r_{Ii}}{r_I}\leq L^2R\eqqcolon \tilde R<1
 \end{equation*}
 which is enough to prove \ref{M:asymptotic}.
	
        Next recall \eqref{usefulbilipestimate} and note that as $\Phi$ is bi-Lipschitz with constant $L$, we have
        \begin{align}\label{eq:1}
          \Phi(M)_{\frac{\delta}{L}}\subseteq \Phi(M_\delta)\subseteq \Phi(M)_{L\delta}
        \end{align}
        for all $\delta>0$ and $M\subseteq X$.
        In the same way
        \begin{align}\label{eq:2}
          \Phi^{-1}(M)_{\frac{\delta}{L}}\subseteq \Phi^{-1}(M_\delta)\subseteq \Phi^{-1}(M)_{L\delta}
        \end{align}
        for all $\delta>0$ and $M\subseteq Y$ hold.
	For the condition \eqref{M:mappings} we consider the bi-Lipschitz mappings
	\begin{equation}
	\psi_i\coloneqq\Phi\circ\varphi_i\circ\Phi^{-1}:Y\to Y.
	\end{equation}
        That implies $\psi_J=\Phi\circ \varphi_J\circ\Phi^{-1}$, and by $\widetilde{K}_J=\Phi(K_J)$ we obtain $\psi_I(\widetilde{K}_J)=\psi_I(\Phi(K_J))=\Phi(\varphi_I(K_J))=\Phi(K_{IJ})=\widetilde{K}_{IJ}$ for all $I,J\in\mathcal{T}^*$.
        
        {\color{black}To prove the last part of \eqref{M:mappings} consider $\widetilde W\coloneqq\frac{W}{L^2}$.
        If $\widetilde G=0$ there is nothing to prove and we can pick any $\widetilde\delta_0>0$.
        In the case $\widetilde G>0$ we have $G>0$ and we can put  $\widetilde\delta_0\coloneqq\frac{\delta_0 GL}{\widetilde G}>0$.
        Then we can write
        \begin{equation*}
        \begin{aligned}
        \psi_I\left(\widetilde{K}_{\widetilde G\delta}\right)&=\Phi\circ \varphi_I\circ\Phi^{-1}\left(\widetilde{K}_{\widetilde G\delta}\right)
        \stackrel{\eqref{eq:1}}{\supseteq}\Phi\circ \varphi_I\left({K}_{\frac{\widetilde G\delta}{L}}\right)\\
        &\stackrel{\eqref{M:mappings}}{\supseteq} \Phi\left(\left({K_I}\right)_{r_I\frac{W\widetilde G\delta}{L}}\right)
        \stackrel{\eqref{eq:1}}{\supseteq} \left(\Phi({K}_I)\right)_{r_I\frac{W\widetilde G\delta}{L^2}}
        =\left(\widetilde{K}_I\right)_{r_I\widetilde W\widetilde G\delta}
        \end{aligned}
        \end{equation*}
        whenever $\delta\leq\widetilde\delta_0$.}
    
        Pick some $\tilde{\xi}>0$ such that
        \begin{equation}\label{eq:choiceOfEta}
        \widetilde G\tilde{\xi}\leq\min\left(\frac{G\xi}{L},\frac{SF }{L\kappa}\right)
        \end{equation}
        where $F$ is the constant from \eqref{eq:oppositeToT2}. Note that this is possible due to the assumptions on $\tilde G$, $G$, and $S$.
        Pick $i$ and $J$ such that $iJ\in\T^*$.
        
         First, by (\ref{eq:1}) one has
        \begin{align*}
          \psi_J(\widetilde{K}_{\widetilde G\tilde{\xi}})=\psi_J(\Phi(K)_{\widetilde G\tilde{\xi}})=\Phi\circ \varphi_J\circ\Phi^{-1}(\Phi(K)_{\widetilde G\tilde{\xi}})\subseteq \Phi\circ\varphi_J(K_{\widetilde GL\tilde{\xi}}).
        \end{align*}
        By (\ref{biLip2}) the mapping $\varphi_J$ is $\kappa^+_{J,\varnothing}r_{J}$-Lipschitz.
        Hence 
         using (\ref{M:kappa}) we obtain
        \begin{align*}
          \varphi_J(K_{\widetilde GL\tilde{\xi}})\subseteq \varphi_J(K)_{\widetilde GL\kappa r_J\tilde{\xi}},
        \end{align*}
        if $K_{\widetilde GL\tilde{\xi}}\subseteq K_{G\xi}$, that is, if the condition
        \begin{align*}
          \widetilde G\tilde{\xi}\le\frac{G\xi}{L}       
        \end{align*}
        holds. But this is true by \eqref{eq:choiceOfEta}.
        Thus for all $x,y\in  \psi_J(\widetilde{K}_{\widetilde G\tilde{\xi}})$ one has
        \begin{align*}
          \Phi^{-1}(x),\Phi^{-1}(y)\in\varphi_J(K_{\widetilde GL\tilde{\xi}})\subseteq \varphi_J(K)_{\widetilde GL\kappa r_J \tilde{\xi}}=(K_J)_{\widetilde GL\kappa r_J\tilde{\xi}}\subseteq (K_J)_{S\diam K_J},
        \end{align*}
        if the condition
        \begin{align*}
         \widetilde G \tilde{\xi}\le \frac{S}{L\kappa}\frac{\diam K_J}{r_Jr}
        \end{align*}
        is satisfied, but this follows from \eqref{eq:oppositeToT2} and \eqref{eq:choiceOfEta}.
        Furthermore, we obtain for all  $x,y\in  \psi_J(\widetilde{K}_{\widetilde G\tilde{\xi}})$ in the same way as above
        \begin{align*}
          \varphi_{i}(\Phi^{-1}(x)),\varphi_{i}(\Phi^{-1}(y))&\in \varphi_i\circ\varphi_J(K_{\widetilde GL\tilde{\xi}})=\varphi_{iJ}(K_{\widetilde GL\tilde{\xi}})\\
          &\subseteq \varphi_{iJ}(K)_{\widetilde GL\kappa r_{iJ}\tilde{\xi}}\subseteq (K_{iJ})_{S\diam K_{iJ}}.
        \end{align*}
        To summarize, one has for all $x,y\in  \psi_J(\widetilde{K}_{\widetilde G\tilde{\xi}})$
        \begin{itemize}
        \item[1.] $x,y\in\Phi((K_J)_{S\diam K_J})$, which leads to the Lipschitz constant $C_{K_J}^{-1}(1+A\;\diam^\alpha K_J)$;
        \item[2.] $\Phi^{-1}(x),\Phi^{-1}(y)\in\varphi_J(K_{ G\xi})$, which leads to the Lipschitz constant $\kappa_{i,J}^+\frac{r_{i,J}}{r_J}$;
          \item[3.] $\varphi_i(\Phi^{-1}(x)),\varphi_i(\Phi^{-1}(y))\in(K_{iJ})_{S\diam K_{iJ}}$, which leads to the Lipschitz constant $C_{K_{ij}}(1+A\;\diam^\alpha K_{iJ})$.
        \end{itemize}
    
	To prove \eqref{M:kappa_to_one} define first $\widetilde \kappa_{i,J}^{\pm}$ as optimal constants in \eqref{biLip1} for $\psi_{i}$ on $\psi_J(\widetilde K_{\widetilde G\tilde{\xi}})$.
	Then with the above observation
	\begin{equation*}
	\begin{aligned}
	&d_{Y}(\psi_i(x),\psi_i(y))=d_{Y}(\Phi\circ\varphi_i\circ\Phi^{-1}(x),\Phi\circ\varphi_i\circ\Phi^{-1}(y))\\
	&\leq \frac{C_{K_{iJ}}}{C_{K_{J}}}\left(1+A\;\diam^\alpha K_{iJ}\right)\left(1+A\;\diam^\alpha K_{J}\right) \frac{r_{iJ}}{r_J} \kappa^{+}_{i,J}  d_{X}(x,y)\\
        &\leq \frac{C_{K_{iJ}}}{C_{K_{J}}}\left(1+A\;(DR^{|iJ|})^\alpha\right)\left(1+A\;(DR^{|J|})^\alpha\right) \frac{r_{iJ}}{r_J} \kappa^{+}_{i,J}  d_{X}(x,y)
	\end{aligned}
	\end{equation*}
        for $x,y\in\psi_J(\widetilde K_{\widetilde G\tilde{\xi}})$.
        Hence one has
        \begin{align*}
        \widetilde \kappa_{i,J}^{+}\frac{\tilde r_{iJ}}{\tilde r_J}\le\frac{C_{K_{iJ}}}{C_{K_{J}}}\frac{r_{iJ}}{r_J}\left(1+A\;(DR^{|J|})^\alpha\right)^2\kappa^{+}_{i,J},
        \end{align*}
        thus
        \begin{align}\label{eq:tildeKappaBound}
          \widetilde \kappa_{i,J}^{+}\le\left(1+A\;(DR^{|J|})^\alpha\right)^2\kappa^{+}_{i,J}.
        \end{align}
        Now, by $0<R<1$ and using $\log t\le t-1$ for $t>0$, we have for all $I,J\in \mathcal{T}^*$ that
        \begin{align*}
        \widetilde \kappa_{I,J}^{+}=\prod\limits_{n=1}^{|I|} \widetilde \kappa^{+}_{I_{n},\sigma^n(I)J}
        &\le\prod\limits_{n=1}^{|I|}  \kappa^{+}_{I_{n},\sigma^n(I)J} \left(1+A\;(DR^{|\sigma^n IJ|})^\alpha\right)^2\\
        &=  \kappa^{+}_{I,J} \prod\limits_{n=1}^{|I|}\left(1+A\;D^\alpha (R^\alpha)^{|I|-n+|J|}\right)^2\\
        &\le  \kappa^{+}_{I,J} \prod\limits_{n=0}^{\infty}\left(1+A\;D^\alpha (R^\alpha)^{n+|J|}\right)^2\\            
        &=\kappa^{+}_{I,J}\exp\left( 2\log \prod\limits_{n=0}^{\infty}\left(1 + A\;D^\alpha (R^\alpha)^{n+|J|}\right)\right)\\
        &\le\kappa^{+}_{I,J}\exp\left( 2A\;D^\alpha \sum\limits_{n=0}^{\infty} (R^\alpha)^{n+|J|}\right)\\        
        &\le\kappa^{+}_{I,J}\exp\left( \frac{2A\;D^\alpha(R^\alpha)^{|J|}}{1 -R^\alpha}\right).
        \end{align*}
        Using \eqref{M:kappa_to_one} for $\kappa_{I,J}^{+}$, the last term converges to $1$ uniformly in $I$, as $|J|\to \infty$.
        That means $\lim_{|J|\to\infty}\widetilde \kappa_{I,J}^{+}\le 1$.
        In the same manner it follows, that $\lim_{|J|\to\infty}\widetilde \kappa_{I,J}^{-}\ge 1$.
        As $\widetilde \kappa_{I,J}^{-}\le \widetilde \kappa_{I,J}^{+}$, we have that $\lim_{|J|\to\infty}\widetilde \kappa_{I,J}^{\pm}=1$.
	
	Applying \eqref{eq:boundsC_K} to $K'\coloneqq K_I$ and $K\coloneqq K_{IJ}$ and using \ref{dva}, \ref{M:asymptotic}, \ref{pet} and ~(\ref{eq:oppositeToT2}), we obtain that there are constants $0<q,Q<\infty$ such that
		\begin{equation}\label{eq:boundsC_I}
		\frac{1+q(\rho^{\alpha})^{|I|+|J|}}{1+Q(R^\alpha)^{|I|}}\leq\frac{1+qr_{IJ}^\alpha}{1+Qr_I^\alpha}\leq \frac{C_{I}}{C_{IJ}}\leq\frac{1+Qr_I^\alpha}{1+qr_{IJ}^\alpha}\leq\frac{1+Q(R^\alpha)^{|I|}}{1+q(\rho^{\alpha})^{|I|+|J|}}.
		\end{equation}
	
	For $\omega\in\T$ define $C_\omega\coloneqq\lim_{n\to\infty} C_{\omega|_n}$.
	We show, that the limit always exists by \eqref{eq:boundsC_I}.
	Indeed, one can rewrite it into 
	\begin{equation*}\label{eq:boundsC_Iminusone}
		\frac{q(\rho^{\alpha})^{|I|+|J|}-Q(R^\alpha)^{|I|}}{1+Q(R^\alpha)^{|I|}}C_{IJ}\leq\left(\frac{C_{I}}{C_{IJ}}-1\right)C_{IJ}\leq\frac{Q(R^\alpha)^{|I|}-q(\rho^{\alpha})^{|I|+|J|}}{1+q(\rho^{\alpha})^{|I|+|J|}}C_{IJ}.
	\end{equation*}
	Since $C_I$ are bounded for all $I\in\T^*$, both left hand side and right hand side converge to $0$ as $|I|\to\infty$ uniformly in $J$.
        Using
	\begin{equation*}
	\left(\frac{C_{I}}{C_{IJ}}-1\right)C_{IJ}=C_I-C_{IJ},
	\end{equation*}
	it follows that $\left(C_{\omega|_n}\right)_{n\ge 0}$ is a Cauchy sequence for each $\omega\in\T$.
        
	Next, observe that \eqref{eq:boundsC_I} also implies
	\begin{equation}\label{eq:boundsC_omega}
		\frac{1}{1+Q(R^\alpha)^{|I|}}\leq \frac{C_{I}}{C_{I\omega}}\leq 1+Q(R^\alpha)^{|I|}
	\end{equation}
	for $I\omega\in\widetilde{\T}\coloneqq\T\cup\T^* $, hence
	\begin{equation}\label{eq:boundsC_omegaTau}
		\frac{1}{\left(1+Q(R^\alpha)^{|I|}\right)^2}\leq \frac{C_{I\tau}}{C_{I\omega}}\leq\left(1+Q(R^\alpha)^{|I|}\right)^2
	\end{equation}
	for $I\omega,I\tau\in\widetilde{\T}$.
	Define a function $\tilde g$ on $\widetilde{\T}$ by $\tilde g(\omega)=\log(C_\omega)$ for $\omega\in\T$ and $\tilde g(I)=\log(C_I)$ for $I\in\T^*$.
        Let $g=\tilde g|_{\T}$.
	We will prove that $\tilde g$ is in fact $\alpha$-H\"older (and so $\tilde g$ is in particular continuous).
	Indeed, applying $\log$ on both sides of both \eqref{eq:boundsC_omega} and \eqref{eq:boundsC_omegaTau} we obtain that there is a constant $0<W<\infty$ such that
	\begin{equation*}
	\left|\tilde g(I\tau)-\tilde g(I\omega)\right|\leq W(R^\alpha)^{|I|}.
	\end{equation*}
	whenever $I\omega,I\tau\in\widetilde{\T}$.
	
	For condition (\ref{M:rho_Jcontinuous}) first observe that, as $\tilde g$ is continuous, also $\omega\mapsto C_\omega$ is continuous.
        Furthermore, for all $\omega\in\T$ one has
	\begin{equation*}
	\frac{\tilde r_{i\omega|_n}}{\tilde r_{\omega|_n}}\to \rho_{i\omega}\frac{C_{i\omega}}{C_\omega}\eqqcolon\tilde{\rho_{i\omega}}
	\end{equation*}
        as $n\to \infty$.
	Also by \eqref{eq:lesserThanL} we obtain that $\tilde\rho_{\omega}\in(0,1)$ for every $\omega\in\T$.
	This gives (\ref{M:rho_Jcontinuous}), here $\omega \to \tilde\rho_{\omega}$ is continuous since both $\omega \to \rho_{\omega}$ and $\omega\to C_\omega$ are.
	The existence of H\"older extension $\tilde f(\omega)=\log\left(\frac{1}{\tilde\rho_\omega}\right)$ is immediate from the H\"olderness of the function $\tilde g$.
	For (\ref{M:holder/nonlattice}), we prove that $\tilde f$ is cohomologous to $f$.
	To do this we can write
	\begin{equation*}
	\begin{aligned}
	f(\omega)-\tilde f(\omega)&=\log\left(\frac{1}{\rho_\omega}\right)-\log\left(\frac{1}{\tilde\rho_\omega}\right)\\
	&=\log\left(\frac{1}{\rho_\omega}\right)-\log\left(\frac{1}{\rho_\omega}\right)+\log(C_{\sigma\omega})-\log(C_\omega)\\
	&=(g\circ\sigma - g)(\omega).
	\end{aligned}
	\end{equation*}
	This completes the proof.
\end{proof}

\subsection{$\C^{1+\alpha}$ images of self-similar sets}
Mappings that are $\alpha$-almost similar are closely related to conformal $\C^{1+\alpha}$ diffeomorphisms
(see the statement of the following proposition for the exact definition). 

\begin{prop}\label{P:C1,qtoqSmooth}
	Let $M\subseteq \er^d$ be an open set and $\Phi:M\to\er^d$ a conformal $\C^{1+\alpha}$ diffeomorphism,
	that is, $\Phi, \Phi^{-1}$ are $\C^{1}$ with $\alpha$-H\"older continuous derivative,
	and the Jacobian matrix of $D\Phi(x)$ is a scalar times a rotation matrix for each $x\in M$.
	Fix a compact subset $\varnothing\neq F\subseteq M$.
	Then there is an $S>0$ such that $\Phi|_F$ is an $\alpha$-almost similarity with parameter $S$. 
\end{prop}
\begin{proof}
	As $\Phi$ is a diffeomorphism and $F$ is compact, we can pick an $S>0$ small enough, such that $F_{S\,\diam F}\subseteq M$
	and $\Phi(F)_{S\,\diam F}\subseteq \Phi(M)$.
	Since $\Phi$ is locally bi-Lipschitz, there is now an $0<L<\infty$ such that
	\begin{equation*}
	\frac{1}{L} \le \frac{|\Phi(x)-\Phi(y)|}{|x-y|}\le L 
	\end{equation*}
	for all $x,y\in F_{S\,\diam F}$.
	Now let $\varnothing\neq K\subseteq F$ be compact.
	Without any restriction we may assume that $\diam F>0$.
	
	Set $\eps\coloneqq S\,\diam F$.
	Then $\eps>0$.
	For $x\in M$, define $f_x:M\to\R$, $z\mapsto \left|\Phi(z)-\Phi(x)\right|$.
	Then
	\[
	 D  f_x(z) = \frac{\Phi(z)-\Phi(x)}{\left|\Phi(z)-\Phi(x)\right|}  \frac{\partial(\Phi_1(z),\dots,\Phi_d(z))}{\partial(z_1,\dots,z_d)}
	\]
	for all $ z\in M\smallsetminus\{x\}$,
	where $(\frac{\partial \Phi_i}{\partial z_j})_{i,j=1}^d$ denotes the Jacobian matrix with respect to the standard coordinates in $\er^d$.
	Again, we may only consider compact $K\subseteq F$ such that $\diam K>0$.
        Fix $x\in F\cap K_{S\,\diam K}$.
	Then $y\in B(x,\eps)\subseteq M$ for all $y\in F\cap K_{S\,\diam K}$.
	Now fix $x\neq y\in F\cap K_{S\,\diam K}$ and set $v\coloneqq y-x/\left| y-x\right|$.
	By the mean value theorem, there exists a $c\in [0,1]$ such that
	\begin{align*}
	\frac{\left|\Phi(y)-\Phi(x)\right|}{\left|y-x\right|}&=\frac{f_x(y)-f_x(x)}{\left|y-x\right|}=\langle D  f_x(cx+(1-c)y),v\rangle\\
	&=\left\langle\frac{\Phi(z)-\Phi(x)}{\left|\Phi(z)-\Phi(x)\right|},\partial_v\Phi(cx+(1-c)y)\right\rangle\\
	&\le \left|\partial_v\Phi(cx+(1-c)y)\right|.
	\end{align*}
	Here in the last step we used the Cauchy-Schwarz inequality.
	Denote $\xi\coloneqq cx+(1-c)y$ and recall, that $D \Phi$ and $D \Phi^{-1}$ are $\alpha$-H\"older with some H\"older constants $H(\Phi)$ and $H(\Phi^{-1})$.
	Define 
	\begin{equation*}
	H\coloneqq \max(H(\Phi),H(\Phi^{-1}),\quad P\coloneqq \inf_{\left|v\right|=1, \xi\in F}\left|\partial_v\Phi(\xi)\right|,
	\quad Q\coloneqq \sup_{\left|v\right|=1, \xi\in F}\left|\partial_v\Phi(\xi)\right|.
	\end{equation*}\
	Note, as $D \Phi\in\C^1$, one has $0<P\le Q<\infty$.
	Hence
	\begin{align*}
	\frac{\left|\Phi(y)-\Phi(x)\right|}{\left|y-x\right|}&\le  \left|\partial_v\Phi(\xi)\right|\le  \left|\partial_v\Phi(x)\right|+\left|\partial_v\Phi(\xi)-\partial_v\Phi(x)\right|\\
	&\le  \left|\partial_v\Phi(x)\right|+H  \left|\xi-x\right|^\alpha\\
	&\le  \left|\partial_v\Phi(x)\right| \left( 1+\frac{H}{ \left|\partial_v\Phi(x)\right|}\diam^\alpha (K_{S\,\diam K})\right)\\
	&\le  \left|\partial_v\Phi(x)\right| \left( 1+HP^{-1}(1+2S)^\alpha\diam^\alpha K\right) 
	\end{align*}
	for all compact $ K\subseteq F$ such that $0<\diam K$,
        and for all $x,y\in F\cap  K_{S\,\diam K}$.
	Define $B\coloneqq(1+2S)^\alpha$ and assume $\diam K_{S\,\diam K}<\frac{\eps}{L}$.
	Then $\diam\Phi(K_{S\,\diam K})\le L \,\diam K_{S\,\diam K}<\eps$.
	As  $\Phi(F)_\eps\subseteq \Phi(M)$, we obtain in the same way as above
	\begin{align*}
	\frac{\left|\Phi^{-1}(s)-\Phi^{-1}(r)\right|}{\left|r-s\right|}\le  \left|\partial_w\Phi^{-1}(r)\right| \left( 1+\frac{H}{ \left|\partial_w\Phi^{-1}(r)\right|}L^\alpha B \diam^\alpha K\right),
	\end{align*}
	where $x,y\in F\cap  K_{S\,\diam K}$, $r\coloneqq \Phi(x)$, $s\coloneqq\Phi(y)$ and $w\coloneqq s-r/\left|s-r\right|$.
	Note, that by conformality of $\Phi$, there exist for each $x\in M$ some $\lambda_x \neq 0$ and rotation matrix $O_x \in\er^{d\times d}$ such that
	$\partial_u\Phi(x)=\lambda_x \,O_x u$ for all $u\in\er^d$.
	Thus $\left|\partial_u\Phi(x)\right|=\left|\lambda_x \right|$ for all $u\in\er^d$ satisfying $\left|u\right|=1$.
	Using the inverse function theorem, this means
	\begin{align*}
	\left|\partial_w\Phi^{-1}(r)\right|=\left|\partial_v  \Phi(x)\right|^{-1}=\left|\lambda_x\right|^{-1},
	\end{align*}
	where $v=y-x/\left| y-x\right|$.
	Thus
	\begin{align*}
	\frac{\left|\Phi(y)-\Phi(x)\right|}{\left|y-x\right|}&= \frac{\left|s-r\right|}{\left|\Phi^{-1}(r)-\Phi^{-1}(s)\right|}\ge\left|\partial_v  \Phi(x) \right|
	(1+H \left|\partial_v  \Phi(x) \right|L^\alpha B \diam^\alpha K)^{-1}\\
	&\ge \left|\partial_v  \Phi(x) \right|  (1+H Q L^{\alpha} B \diam^\alpha K)^{-1}.
	\end{align*}
	
	Therefore we obtain for all compact $ K\subseteq F$ such that $\diam K_{S\,\diam K}<\frac{\eps}{L}$ and all $x\neq y\in F\cap  K_{S\,\diam K}$
	\begin{align*}
	\left|\lambda_x \right|  (1+H Q L^{\alpha} B \diam^\alpha K)^{-1}\le \frac{\left|\Phi(y)-\Phi(x)\right|}{\left|y-x\right|}\le \left|\lambda_x \right|  (1+H P^{-1} B \diam^\alpha K).
	\end{align*}
	Define
	\begin{align*}
	\gamma\coloneqq\inf\left\{\diam K:K\subseteq F\textrm{ compact, }\diam K_{S\,\diam K}\ge\frac{\eps}{L}\right\}.
	\end{align*}
	If there are compact $K\subseteq F$ such that $\diam K_{S\,\diam K}\ge\frac{\eps}{L}$, then $0<\gamma<\infty$, else $\gamma=\inf\varnothing=\infty$.
	Finally define $A\coloneqq  \max(H Q L^{\alpha}B,HP^{-1}B,\gamma^{-\alpha}(L-1))$ and fix an $x_K\in K$ for each compact $\varnothing\neq K\subseteq F$.
	Then setting $C_K\coloneqq \left|\lambda_{x_K} \right|$ if $\diam K_{S\,\diam K}<\frac{\eps}{L}$, and $C_K\coloneqq 1$ if $\diam K_{S\,\diam K}\ge\frac{\eps}{L}$,
	completes the proof.
\end{proof}

\begin{rem}
	Above statement still holds, if one assumes the Jacobian matrix of $D\Phi(x)$ to be scalar times orthogonal matrix for each $x\in M$.
\end{rem}


\section*{Acknowledgements}

We want to thank Martina Z\"ahle for her support
and giving us the opportunity to work on Ahlfors regular sets.
Also we want to thank Jan Rataj for useful discussions concerning
Lemma \ref{rataj-lemma}, and Sabrina Kombrink for helpful
comments concerning the issue of non-latticeness.

\bibliographystyle{abbrv} \bibliography{references}
\end{document}